\documentclass[leqno]{amsart}
\usepackage{amsfonts,amssymb,amsmath,amsgen,amsthm}
\usepackage{hyperref,color}
\usepackage{pdfsync}
\usepackage{bbm}
\usepackage{bm}

\theoremstyle{plain}
\newtheorem{theorem}{Theorem}[section]

\newtheorem{lemma}[theorem]{Lemma}

\newtheorem{proposition}[theorem]{Proposition}

\theoremstyle{remark}
\newtheorem{remark}[theorem]{Remark}

\def\C{{\mathbb C}}
\def\R{{\mathbb R}}
\def\N{{\mathbb N}}

\def\({\left(}
\def\){\right)}
\def\<{\left\langle}
\def\>{\right\rangle}

\def\1{{\mathbf 1}}

\def\d{{\partial}}
\def\eps{\varepsilon}

\DeclareMathOperator{\RE}{Re}
\DeclareMathOperator{\IM}{Im}
\DeclareMathOperator{\diver}{div}

\numberwithin{equation}{section}

\date\today

\title[GWP for the non-linear Maxwell-Schr\"odinger system]{Global well-posedness for the non-linear Maxwell-Schr\"odinger system}

\author[P.~Antonelli]{Paolo Antonelli}
\address[P.~Antonelli]{Gran Sasso Science Institute \\ via Crispi 7 \\ 67100 L'Aquila (Italy).}
\email{paolo.antonelli@gssi.it}
\author[P.~Marcati]{Pierangelo Marcati}
\address[P.~Marcati]{Gran Sasso Science Institute \\ via Crispi 7 \\ 67100 L'Aquila (Italy).}
\email{pierangelo.marcati@gssi.it}
\author[R.~Scandone]{Raffaele Scandone}
\address[R.~Scandone]{Gran Sasso Science Institute \\ via Crispi 7 \\ 67100 L'Aquila (Italy).}
\email{raffaele.scandone@gssi.it}

\begin{document}

\begin{abstract}
In this paper we study the Cauchy problem associated to the Maxwell-Schr\"odinger system with a defocusing pure-power non-linearity. 
This system has many applications in physics, for instance in the description of a charged non-relativistic quantum plasma, interacting with its self-generated electromagnetic potential. 

One consequence of our analysis is to demonstrate that the Lorentz force associated with the electromagnetic field is well-defined for solutions slightly more regular than the finite energy class. This aspect is of fundamental importance since all the related physical models require the observability of electromagnetic effects. The well-posedness of the Lorentz force still seems to be a major open problem in the class of solutions which are only finite energy.

We show the global well-posedness at high regularity for the cubic and sub-cubic case, and we provide polynomial bounds for the growth of the Sobolev norm of the solutions, for a certain range of non-linearities. An important role is played by appropriate a priori dispersive estimates, obtained by means of Koch-Tzvetkov type bounds for the non-homogeneous Schr\"odinger equation, which overcome the lack of Strichartz estimates for the magnetic Schr\"odinger flow.

  Because of the power-type non-linearity, the propagation of higher regularity, globally in time, cannot be achieved via a bootstrap argument as done in \cite{Nakamura-Wada-MaxwSchr_CMP2007}. 
   Our approach then exploits the analysis of a modified energy functional, combined with the a priori bounds coming from the dispersive estimates obtained previously.
 

\end{abstract}
\maketitle

\section{Introduction}
In this paper we study the following non-linear Maxwell-Schr\"odinger system in the Coulomb gauge
\begin{equation}\label{eq:MS}
\begin{cases}
i\partial_t u = -\Delta_Au+\phi u+|u|^{\gamma-1}u\\
\square A=\mathbb{P}J
\end{cases}\quad t\in\R,x\in\R^3,
\end{equation}
in the unknown $(u,A):\R_t\times\R^3\to\C\times\R^3$, with initial conditions
$$(u(0),A(0),\partial_tA(0))=(u_0,A_0,A_1),\quad\diver A_0=\diver A_1=0,$$
where $\Delta_A:=(\nabla-iA)^2$ is the magnetic Laplacian, $\phi=\phi(u):=(-\Delta)^{-1}|u|^2$, $J=J(u,A):=2Im (\overline{u}(\nabla-iA)u)$, and $\mathbb{P}:=\mathbb{I}-\nabla\diver\Delta^{-1}$ is the Helmholtz-Leray projection onto divergence free vector fields.

Physically $u$ can be interpreted as the order parameter associated to a charged quantum plasma \cite{SE,SE3,Haas2,SE2}, interacting with its self-generated electromagnetic potential described by $(\phi,A)$. Moreover, $\rho:=|u|^2$ and $J(u,A)$ are, respectively, the charge and the electric current density. The power-type non-linearity is introduced in order to encode pressure effects \cite{SE3}, see also the discussion later in this introduction.
Formally, the charge $\mathcal{Q}(t):=\|\rho\|_{L^1}=\|u\|_{L_2}^2$, and the energy
\begin{equation}\label{def:energy}
\mathcal{E}(t):=\int_{\R^3}|(\nabla-iA)u|^2+\frac12\big(|\partial_t A|^2+|\nabla A|^2+|\nabla\phi|^2\big)+\frac{2}{\gamma+1}|u|^{\gamma+1}dx,
\end{equation}
are conserved by solutions to \eqref{eq:MS}.

The system \eqref{eq:MS} is strictly related to the classical Maxwell-Schr\"odinger system
\begin{equation}\label{clas_MS}
\begin{cases}
i\partial_t u=-\Delta_A u + \phi u\\
-\Delta\phi-\partial_t\diver A=\rho\\
\square A+\nabla(\partial_t\phi+\diver A)=J,
\end{cases}
\end{equation}
in the unknown $(u,\phi,A):\R_t\times\R^3\to\C\times\R\times\R^3$, which describes the dynamics of a charged non-relativistic quantum particle, subject to its self-generated (classical) electro-magnetic field \cite{Schi,Fey}. In particular, \eqref{clas_MS} provides a classical approximation to the quantum field equations for an electro-dynamical non-relativistic many-body system.

It is well-known that \eqref{clas_MS} is invariant under the gauge transformation
\begin{equation}\label{eq:gauge}
(u,\phi,A)\mapsto (e^{i\lambda}u,\phi-\partial_t\lambda,A+\nabla\lambda).
\end{equation}
In particular, in the Coulomb gauge, i.e.~$\diver A=0$, it takes the simple form
\begin{equation}\label{eq:MSC}
\begin{cases}
i\partial_t u = -\Delta_Au+\phi u\\
\square A=\mathbb{P}J,
\end{cases}
\end{equation}
where $\phi$ is explicitly given by $(-\Delta)^{-1}|u|^2$. 

The Cauchy problem associated with the Maxwell-Schr\"odinger system \eqref{clas_MS} has been widely studied in the mathematical literature, under various choices of the gauge. Among the first treatments we mention \cite{Nakamitsu-Tsutsumi-JMP,Tsutsumi}, where the authors studied the local and global well-posedness in high regularity spaces by means of the Lorenz gauge. The global existence of finite energy weak solutions has been investigated in \cite{Guo-Nakamitsu-Strauss}, by using a vanishing viscosity approach. The asymptotic behavior and the long-range scattering of solutions to \eqref{clas_MS} has been studied in \cite{Ginibre-Velo_S1,Ginibre-Velo_S2,Shimomura} (see also the references therein).

In \cite{Nakamura-Wada-Local,Nakamura-Wada-MaxwSchr_CMP2007}, using the evolution semigroup associated to the magnetic Laplacian, the authors obtained global well-posedness at high Sobolev regularity by means of a fixed point argument and suitable a priori estimates. The well-posedness at low regularity, and in particular in the energy space, can not be easily handled with these techniques, due to the difficulty to construct the linear magnetic propagator. The question has been recently solved in \cite{Bejenaru-Tataru}, by using the analysis of a short time wave packet parametrix for the magnetic Schr\"odinger equation and the related linear, bilinear, and trilinear estimates. 
On the other hand, at present it is still not clear whether the finite energy framework provides the sufficient regularity needed in order to define the Lorentz force. This aspect has its relevance since all related physical models require the observability of electromagnetic effects \cite{Fey}. This issue is straightforwardly overcome in the higher regularity framework, as for instance in \cite{Nakamura-Wada-Local, Nakamura-Wada-MaxwSchr_CMP2007}, and it turns out to be solved also for solutions slightly more regular than finite energy, see Proposition \ref{orc} below.
\newline
Let us remark that the Maxwell-Schr\"odinger system \eqref{clas_MS}, even when augmented with a power type nonlinearity as in \eqref{eq:MS}, retains its invariance with respect to gauge transformations. Our choice of studying problem \eqref{eq:MS} in the Coulomb gauge is only dictated by the aim of providing a simpler presentation. Indeed, by using a similar strategy the problem could be studied also in the Lorenz gauge, for instance, see Section \ref{sec:lor} for more details. On the other hand, it is possible to translate part of the results, studied in the Coulomb gauge, also to other gauges by using suitable transformations, see e.g. the discussion in \cite[Section 5]{Guo-Nakamitsu-Strauss}.
%
%

As already mentioned, the study of the non-linear Maxwell-Schr\"odinger system is motivated by the physical applications of this model. Indeed, system \eqref{eq:MS} arises in the description of dense astrophysical plasmas exhibiting quantum effects \cite{Haas,Haas2,SE,SE3,SE2}. More precisely, by means of the Madelung transform \cite{Mad} and by identifying $\rho$ and $J$ with the hydrodynamical momenta associated to $u$, it is possible to draw an analogy between \eqref{eq:MS} and a hydrodynamical system describing a compressible, inviscid, charged, quantum fluid modeling a quantum plasma \cite{SE3}. By exploiting this analogy, the nonlinear term in \eqref{eq:MS} then corresponds to the so called electron degeneracy pressure in the fluid dynamic description. The interested reader could refer to Section III in \cite{SE3} for more details about the physical modeling and to \cite{Antonelli-Marcati, AHMZ, AM_Brix} and the references therein for the rigorous setting of the Madelung transform.

From the mathematical point of view, the power-type non-linearity in \eqref{eq:MS} introduces further difficulties. The lack of suitable space-time estimates for the classical Maxwell-Schr\"odinger system prevents the study of \eqref{eq:MS} as a perturbation of \eqref{eq:MSC}. For instance, the analysis in \cite{Bejenaru-Tataru} cannot be straightforwardly adapted to the non-linear case. In order to deal with the power-type non-linearity at low-regularity regimes, one would need some kind of global smoothing properties for the magnetic-Schr\"odinger flow, such as Strichartz estimates. Although magnetic Strichartz estimates are well understood for time independent potentials \cite{DAncona-Fanelli-StriSmoothMagn-2008,DAFVV-2010,EGSchlag-2008,EGSchlag-2009, FV} (see also \cite{Michelangeli} and references therein), in the time dependent case much less is known, and the only results available require the smallness of suitable scale invariant space-time norms \cite{Georgiev-Stefanov-Tarulli-2007,Stefanov-MagnStrich-2007}. In particular, even for a non-linear Schr\"odinger equation with a given external time dependent magnetic potential, the well-posedness in the energy space is in general an open question (see \cite[Chapter 4]{Andreas} for relevant advances in this direction, and the paper \cite{AMS-2017-globalFinErgNLS} for the existence of finite energy weak solutions). Concerning the Maxwell-Schr\"odinger system with focusing non-linearities, one can also study the existence and stability of standing waves, see for instance \cite{CW} and references therein.

A convenient regularity framework for \eqref{eq:MS} is given by \cite{Nakamura-Wada-Local,Nakamura-Wada-MaxwSchr_CMP2007} (see also \cite[Chapter 2]{Andreas}), where the authors determine the sufficient regularity in order to construct the evolution semigroup associated to the magnetic Laplacian. On the other hand, in this framework it is not possible to use standard arguments such as the conservation of charge and energy in order to extend the solution globally in time, see for example \cite{Antonelli-Damico-Marcati,CW2020} where the estimates inferred are not sufficient to control the non-linearity globally in time.

To overcome those difficulties here we combine two main ingredients. First of all, we derive suitable a priori estimates (encoded in Propositions \ref{th:apriori} and \ref{th:perdi}) for weak solutions to the non-linear Maxwell-Schr\"odinger system \eqref{eq:MS}, in the same spirit as in \cite{Nakamura-Wada-MaxwSchr_CMP2007}. The relevant tools are the Strichartz estimates for the Klein-Gordon equation, and the smoothing-Strichartz estimates for the inhomogeneous Schr\"odinger equation, which allow us to deal with the derivative term in the expansion of the magnetic Laplacian. The a priori estimates, combined with the conservation of charge and energy, imply a non-trivial gain of spatial integrability, see Lemma \ref{ove_pi}. 
The second ingredient is a classical argument involving modified energies (see for instance the papers \cite{Tsu,Raphael-S,Ozawa-Visciglia}), which improves on the standard energy method for $\gamma>2$. Combining this argument with the a priori estimates allows us to deduce global well-posedness at the $H^2$-regularity level for the order parameter, in the regime $\gamma\in(1,3]$.  Moreover, we obtain a polynomial bound for the growth of the Sobolev norms of the solutions when $\gamma\in(2,3)$, and an exponential bound in the cubic case $\gamma=3$ (see the recent paper \cite{P-T-V} for an analogous result in the case of NLS on compact manifolds).


Let us state our main result. For $s,\sigma\in\R$ we set 
$$\Sigma^{\sigma}:=\{(A_0,A_1)\in H^{\sigma}(\R^3,\R^3)\times H^{\sigma-1}(\R^3,\R^3)\,|\,\diver{A_0}=\diver A_1=0\},$$
and $M^{s,\sigma}:=H^{s}(\R^3,\C)\times \Sigma^{\sigma}$. We have the following theorem.

\begin{theorem}\label{th:main}
Fix $\gamma\in(1,3]$, $\sigma\in[\frac43,3)$. The Cauchy problem associated with \eqref{eq:MS} is globally well-posed in $M^{2,\sigma}$. Namely, for any initial data $(u_0,A_0,A_1)\in M^{2,\sigma}$, there exists a unique solution $(u,A)$ to \eqref{eq:MS}, with $(u,A,\partial_t A)\in\mathcal{C}([0,+\infty),M^{2,\sigma})$. Moreover
\begin{itemize}
\item[(i)] There is continuous dependence on the initial data. Namely, for every $0<T<\infty$, the flow map $(u_0,A_0,A_1)\mapsto (u,A,\partial_t A)$ is continuous from $M^{2,\sigma}$ to $\mathcal{C}([0,T],M^{2,\sigma})$.
\item[(ii)] The charge and the energy are conserved, i.e.~$\mathcal{Q}(t)=\mathcal{Q}(0)$, $\mathcal{E}(t)=\mathcal{E}(0)$ for every $t>0$.
\end{itemize}
In addition, when $\gamma\in(2,3)$ there exists a positive constant $N$, independent on $\gamma$, such that
\begin{equation}\label{pol_bound}
\|(u,A,\partial_t A)\|_{L^{\infty}((0,T);M^{2,\sigma})}\lesssim (1+T)^{\frac{N}{3-\gamma}}.
\end{equation}
When $\gamma=3$ we have instead the exponential bound
\begin{equation}\label{exp_bound}
\|(u,A,\partial_t A)\|_{L^{\infty}((0,T);M^{2,\sigma})}\lesssim \exp(T^{N}),
\end{equation}
for some constant $N>0$.
\end{theorem}

Theorem \ref{th:main} improves on the results in \cite{Antonelli-Damico-Marcati} and \cite{CW2020}, where the authors proved local well-posedness at high-regularity regimes. Moreover, it is worth noticing that a similar global-in-time result can be also obtained in the focusing case, for the mass sub-critical regime $\gamma\in(1,\frac73)$. For the sake of concreteness, here we focus on the defocusing case only.

The paper is organized as follows. In Section 2, we collect the main tools we use throughout the paper, in particular the Strichartz estimates for the Klein-Gordon equation, and the smoothing-Strichartz estimates for the inhomogeneous Schr\"odinger equation. Section 3 is devoted to the proof of the a priori estimates for weak solutions to \eqref{eq:MS}, which are the key tool for the globalization argument. Moreover, they allow to show that the Lorentz force is well defined for weak solutions which are slightly more regular than just finite energy, see Proposition \ref{orc} and the subsequent remark. In Section 4, owing to the theory on the linear magnetic-Schr\"odinger propagator, we prove, by means of a contraction argument, local well-posedness for \eqref{eq:MS}, for every $\gamma>1$ (Theorem \ref{th:LWP}). In Section 5, we show the global well-posedness in the cubic and sub-cubic case. For $\gamma\in(1,2]$, the proof is based on a standard energy method, combined with the Brezis-Gallouet-Wainger inequality and a Gr\"onwall-type argument. For $\gamma\in(2,3]$, we exploit the properties of the modified energy, which allow us to obtain also the polynomial bound \eqref{pol_bound} in the sub-cubic case, and the exponential bound \eqref{exp_bound} in the cubic case.

\section*{Acknowledgments}
The authors are grateful to Andreas Geyer-Schulz for some useful discussions. The authors acknowledge partial support by INdAM-GNAMPA through the project ``Esistenza, limiti singolari e comportamento asintotico per equazioni Eulero/Navier-Stokes-Korteweg”.

\section{Notation and preliminaries}
In this Section we collect some preliminary results we are going to use throughout the paper. We begin with a few remarks on our notation. 

We often write $L^p$ (resp.~$W^{s,p}$) to denote the Lebesgue space $L^p(\R^3)$ (resp.~the Sobolev space $W^{s,p}(\R^3)$). As usual, $H^s$  denotes the space $W^{s,2}$. For any interval $I\subseteq\R$ and any Banach space $\mathcal{X}$, we denote by $L^p(I,\mathcal{X})$ (resp.~$W^{s,p}(I,\mathcal{X})$) the space of $\mathcal{X}$-valued Bochner measurable function on $I$, whose $\mathcal{X}$-norm belongs to $L^p(I)$ (resp.~$W^{s,p}(I)$). These spaces will be often abbreviated to $L_T^p\mathcal{X}$ and $W_T^{s,p}\mathcal{X}$ when $I=[0,T]$. Given $p\geq 1$, we denote by $p'$ its dual exponent. As customary, we set $\langle \lambda\rangle:=\sqrt{1+\lambda^2}$ for $\lambda\in\R$. Given two positive quantities $A,B$, we write $A\lesssim B$ if there exists a constant $C>0$ such that $A\leq CB$; if $A$ and $B$ depend on a positive parameter $T$, we write $A\lesssim_{\langle T^n\rangle}B$ if $A\lesssim \langle T\rangle^n B$ for some positive constant $n$. We denote by $(\cdot,\cdot)$ the standard scalar product on $L^2$. For a given vector field $A:\R^3\to\R^3$, we define the magnetic gradient $\nabla_A:=(\nabla-iA)$. Given $s\in\R$, we write $\mathcal{D}^s:=(1-\Delta)^{s/2}$ for the Bessel operator of order $s$ (we just write $\mathcal{D}$ when $s=1$). When not specified otherwise, $m$ denotes a positive integer constant, which may change at each occurrence.

We recall the generalized fractional Leibniz rule \cite{Gulisashvili-Kon-1996}.

\begin{lemma}
Let  $s,\,\alpha,\,\beta\in[0,\infty)$, $p\in(1,\infty)$, and let $p_1,\,p_2,\,q_1,\,q_2\in(1,\infty]$ be such that $\frac{1}{p_i}+\frac{1}{q_i}=\frac1p$, $i=1,2$. Then
\begin{equation}\label{frac_leibniz}
\|\mathcal{D}^s(fg)\|_{L^p}\lesssim \|\mathcal{D}^{s+\alpha}f\|_{L^{p_1}}\|\mathcal{D}^{-\alpha}g\|_{L^{q_1}}+\|\mathcal{D}^{-\beta}f\|_{L^{p_2}} \|\mathcal{D}^{s+\beta}g\|_{L^{q_2}}.
\end{equation}
\end{lemma}

We shall also use the following estimate, which can be deduced by the Kato-Ponce commutator estimates \cite{Kato-Ponce_CommEst-1988} and the observation that $\mathbb{P}\nabla=0$ (see \cite{Nakamura-Wada-Local} for details).

\begin{lemma}\label{le:tre}
Let $s\geq 0$, and let $p,p_1,p_2\in(1,\infty)$, $q_1,q_2\in(1,\infty]$ be such that $\frac{1}{p_i}+\frac{1}{q_i}=\frac1p$, $i=1,2$. Then
$$\|\mathbb{P}(\bar{f}_1\nabla f_2)\|_{W^{s,p}}\lesssim\|f_1\|_{W^{s,p_1}}\|\nabla f_2\|_{L^{q_1}}+\|\nabla f_1\|_{L^{q_2}}\|f_2\|_{W^{s,p_2}}.$$
\end{lemma}

Let us recall the Brezis-Gallouet-Wainger inequality \cite{Bre-G,Bre-W}.

\begin{lemma}\label{le:cinque}
Let $p,q\in(1,\infty)$ and $\alpha>0$. We have the estimate
\begin{equation}\label{brez}
\|f\|_{L^{\infty}}\lesssim 1+\|f\|_{W^{3/p,p}}\ln^{(p-1)/p}(e+\|f\|_{W^{3/q+\alpha,q}}).
\end{equation}
\end{lemma}

The following two lemmas will be useful when we estimate the non-linear terms $\phi u$ and $|u|^{\gamma-1}u$ in fractional Sobolev spaces. They can be deduced, respectively, from \cite[Lemma 2.1]{Nakamura-Wada-MaxwSchr_CMP2007} and \cite[Proposition 4.9.4]{cazenave}.

\begin{lemma}\label{esti_conv}
For every $s\in[0,1]$ we have
$$\|\big((-\Delta)^{-1}|u|^2\big)u\|_{H^s}\lesssim\|u\|_{H^1}^2\|u\|_{H^{s+1}}.$$
\end{lemma}

\begin{lemma}\label{esti_pure}
Let $\gamma>1$, $s\in(0,1)$, $\eps>0$, and let $p,\widetilde{p}\in[2,\infty]$ be such that $\frac{1}{p}+\frac{1}{\widetilde{p}}=\frac12$. Then
$$\||u|^{\gamma-1}u\|_{H^s}\lesssim\|u\|_{L^{p(\gamma-1)}}^{\gamma-1}\|u\|_{W^{s+\eps,\widetilde{p}}}.$$
\end{lemma}

Next, we state the Strichartz estimates for the Klein-Gordon equation. We say that a pair $(q,r)$ is Klein-Gordon admissible if $\frac{1}{q}+\frac{1}{r}=\frac12$, $q\in(2,+\infty]$. We have the following result \cite{Brenner,Ginibre-Velo_KG,Ginibre-Velo_Strichartz}.
\begin{lemma}\label{le:uno}
Let $T>0$, $s\in\R$, and let $(q_0,r_0)$, be a Klein-Gordon admissible pair. For any given $(A_0,A_1)\in\Sigma^s$ and $F\in L_{T}^{q'_0}W^{s-1+2/q_0,r'_0}$, there exists a unique solution  $A\in C([0,T],H^s)\cap C^1([0,T],H^{s-1})$ to the equation $(\square+1)A=F$, with initial data $A(0)=A_0$, $\partial_tA(0)=A_1$. Moreover, for every Klein-Gordon admissible pair $(q,r)$, we have the estimate
\begin{equation}\label{eq:kg_stri}
\max_{k=0,1}\|\partial_t^k A\|_{L_T^{q}W^{s-k-2/q,r}}\lesssim\|(A_0,A_1)\|_{\Sigma^s}+\|F\|_{L_T^{q'_0}W^{s+2/q_0-1,r'_0}}.
\end{equation}
\end{lemma}

We also need a suitable smoothing-Strichartz estimate for the inhomogeneous Schr\"odinger equation. We recall that a pair $(q,r)$ is Schr\"odinger admissible if $\frac{2}{q}+\frac{3}{r}=\frac32$, $q\in[2,+\infty]$. We have the following result \cite{Nakamura-Wada-MaxwSchr_CMP2007}. 

\begin{lemma}\label{le:due}
Let $T>0$, $s,\alpha\in\R$, and let $(q,r)$ be a Schr\"odinger admissible pair. Let $F\in L_T^{2}H^{s-2\alpha}$, and let $u\in L^{\infty}H^s$ be a weak solution to $i\partial_t u=-\Delta u + F$. Then $u$ satisfies
\begin{equation}\label{kt_stri}
\|u\|_{L_T^qW^{s-\alpha,r}}\lesssim \|u\|_{L_T^{\infty}H^s}+T^{1/2}\|F\|_{L_T^2H^{s-2\alpha}}.
\end{equation}
\end{lemma}

This kind of estimate was originally proved by Koch-Tzvetkov \cite{Koch-Tzvetkov} and Kenig-Koenig \cite{Kenig-Koenig} for the Benjamin-Ono equation. In \cite{Kato-Sch-maps} they were adapted to the Schr\"odinger equation, with an $\eps$-loss of regularity, and finally proved by Nakamura-Wada \cite{Nakamura-Wada-MaxwSchr_CMP2007} in the form above. 



Applying Lemma \ref{le:due}, with $\alpha\geq\frac12$, to the linear magnetic Schr\"odinger equation, it is possible to control the derivative term $A\nabla u$, provided the magnetic potential is regular enough. More precisely, we have the following result.

\begin{lemma}\label{le:quattordici}
Let $T>0$, $s\in [1,2]$, $\alpha\in[\frac12,1)$, $\sigma\geq 1$, with $(\alpha,\sigma)\neq(\frac12,1)$. Let $A\in L_T^{\infty}H^{\sigma}\cap L_T^2L^{3/(2\alpha-1)}$, with $\diver A=0$, and $F\in L_T^{2}H^{s-2\alpha}$. Then a weak solution $u$ to $i\partial_t u=-\Delta_A u + F$ satisfies
\begin{equation}\label{ext:nw}
\|u\|_{L_T^2W^{s-\alpha,6}} \lesssim_{\langle T\rangle^n}\langle\|A\|_{ L_T^{\infty}H^{\sigma}\cap L_T^2L^{3/(2\alpha-1)}}\rangle^m\|u\|_{L_T^{\infty}H^s}+\|F\|_{L_T^2H^{s-2\alpha}}.
\end{equation}
\end{lemma}

\begin{proof}
The case $\alpha=1/2$ has been proved in \cite[Lemma 3.1]{Nakamura-Wada-MaxwSchr_CMP2007}. Let us focus on the case $\alpha>\frac12$.
Expanding the magnetic Laplacian, and applying Lemma \ref{le:due} with the endpoint Strichartz pair $(q,r)=(2,6)$, we get
\begin{equation}\label{eq:llez}
\|u\|_{L_T^2W^{s-\alpha,6}}\lesssim \|u\|_{L_T^{\infty}H^s}+T^{1/2}\|A\nabla u+|A|^2u+F\|_{L_T^2H^{s-2\alpha}}.
\end{equation}
We start by estimating the term $A\nabla u$. When $s\in[1,2\alpha]$, Sobolev embedding and H\"older inequality yield
\begin{equation}\label{bic_stock}
\begin{split}
\|A\nabla u\|_{L_T^2H^{s-2\alpha}}&\lesssim\|A\|_{L_T^2L^{3/(2\alpha-1)}}\|\nabla u\|_{L_T^{\infty}L^{6/(5-2s)}}\\
&\lesssim\|A\|_{L_T^2L^{3/(2\alpha-1)}}\|u\|_{L_T^{\infty}H^s}.
\end{split}
\end{equation}
Consider now the case $s\in(2\alpha,2]$. At spatial level, the fractional Leibniz rule \eqref{frac_leibniz} and Sobolev embedding give
\begin{equation}\label{eq:canfo}
\begin{split}
\|A\nabla u\|_{H^{s-2\alpha}}&\lesssim \|A\|_{L^{3/(2\alpha-1)}}\|u\|_{W^{s+1-2\alpha,6/(5-4\alpha)}}+\|A\|_{W^{s-2\alpha,r}}\|\nabla u\|_{L^q}\\
&\lesssim \|A\|_{L^{3/(2\alpha-1)}}\|u\|_{H^s}+\|A\|_{W^{s-2\alpha,r}}\|\nabla u\|_{L^q},
\end{split}
\end{equation}
for every Klein-Gordon admissible pair $(q,r)$. In particular, we choose the pair $(q,r)=(q(s,\alpha),r(s,\alpha))$ given by
$$
\frac2r=1-\frac2q:=(s-2\alpha)+\frac23(2\alpha-1)(1+2\alpha-s).
$$
Using Sobolev embedding and Gagliardo-Nirenberg interpolation inequality, we find $\theta:=\theta(s,\alpha)\geq\frac2q$ such that
\begin{equation}\label{ind_sic}
\|\nabla u\|_{L^q}\lesssim \|u\|^{\theta}_{H^s}\|u\|_{W^{s-\alpha,6}}^{1-\theta}.
\end{equation}
The bounds \eqref{eq:canfo}-\eqref{ind_sic}, together with Young inequality and H\"older inequality in the time variable yield
\begin{equation}\label{eq:hisp}
\begin{split}
\|A\nabla u\|_{L_T^2H^{s-2\alpha}}&\lesssim \|A\|_{L_T^2L^{3/(2\alpha-1)}}\|u\|_{L_T^{\infty}H^s}\\
&\qquad +\|A\|_{L_T^{2/\theta}W^{s-2\alpha,r}}\|u\|_{L_T^{\infty}H^s}^{\theta}\|u\|_{L_T^2W^{s-\alpha,6}}^{1-\theta}\\
&\lesssim\langle T\rangle^m\big(\|A\|_{L_T^2L^{3/(2\alpha-1)}}+\eps^{1-1/\theta}\|A\|_{L_T^qW^{s-2\alpha,r}}^{1/\theta}\big)\|u\|_{L_T^{\infty}H^s}\\
&\qquad +\eps\|u\|_{L_T^2W^{s-\alpha,6}},
\end{split}
\end{equation}
for any $\eps>0$. Moreover, for a suitable $c\in(0,1)$, we have the interpolation inequality
$$\|A\|_{W^{s-2\alpha,r}}\lesssim \|A\|_{H^1}^{c}\|A\|_{L^{3/(2\alpha-1)}}^{1-c},$$
which implies
\begin{equation}\label{eq:ita_control}
\|A\|_{L_T^qW^{s-2\alpha,r}}\lesssim_{\langle T\rangle^n}\|A\|_{L_T^{\infty}H^{\sigma}\cap L_T^2L^{3/(2\alpha-1)}}.
\end{equation}
Combining \eqref{eq:hisp} and \eqref{eq:ita_control}, we obtain
\begin{equation}\label{eq:llep}
\|A\nabla u\|_{L_T^2H^{s-2\alpha}}\lesssim_{\langle T\rangle^n}\langle\|A\|_{ L_T^{\infty}H^{\sigma}\cap L_T^2L^{3/(2\alpha-1)}}\rangle^m\|u\|_{L_T^{\infty}H^s}+\eps\|u\|_{L_T^2W^{s-\alpha,6}},
\end{equation}
which in view of \eqref{bic_stock} is valid in the whole regime $s\in[1,2]$.

Next, we consider the term $|A|^2u$. Let us prove the estimate
\begin{equation}\label{eq:lles}
\||A|^2u\|_{L_T^2H^{s-2\alpha}}\lesssim_{\langle T\rangle^n}\|A\|^2_{L^{\infty}_TH^{\sigma}\cap L_T^2L^{3/(2\alpha-1)}}\|u\|_{L_T^{\infty}H^s}.
\end{equation}
When $s\in[1,2\alpha]$, \eqref{eq:lles} follows by Sobolev embedding and H\"older inequality. When $s\in(2\alpha,2]$, we use the fractional Leibniz rule \eqref{frac_leibniz} to deduce
\begin{equation*}
\begin{split}
\||A|^2u\|_{H^{s-2\alpha}}&\lesssim \|A\|_{W^{s-2\alpha,6/(2s+1-4\alpha)}}\|A\|_{L^{3/(2\alpha-1)}}\|u\|_{L^{3/(2-s)}}+\|A\|^2_{L^6}\|u\|_{W^{s-2\alpha,6}}\\
&\lesssim \|A\|_{H^{\sigma}}(\|A\|_{L^{3/(2\alpha-1)}}+\|A\|_{H^{\sigma}})\|u\|_{H^s},
\end{split}
\end{equation*}
and \eqref{eq:lles} immediately follows.

Finally, estimate \eqref{ext:nw} is proved by combining \eqref{eq:llez}, \eqref{eq:llep} and \eqref{eq:lles}, and by choosing $\eps>0$ sufficiently small so that $\eps\|u\|_{L_T^2W^{s-\alpha,6}}$ can be absorbed into the left hand side of the inequality.
\end{proof}

We conclude this Section with some useful results for time independent magnetic potentials. For any given $A\in L^2_{\mathrm{loc}}(\R^3)$, the magnetic Laplacian $-\Delta_A$ can be defined as a non-negative self-adjoint operator on $L^2(\R^3)$, by means of a quadratic form argument \cite{Simon-79_maximal-minimal_Schr_forms}.
Given $s\geq 0$, we can define the magnetic Sobolev space 
$$H_A^s(\R^3):=\mathcal{D}((-\Delta_A+1)^{s/2}),\quad\|f\|_{H_A^s(\R^3)}:=\|(-\Delta_A+1)^{s/2}f\|_{L^2(\R^3)}.$$

When the magnetic potential is regular enough, the classical and magnetic Sobolev norms, for a suitable regime of regularity, are equivalent. In particular, we shall use the following result (see e.g.~\cite[Lemma 2.2]{Nakamura-Wada-MaxwSchr_CMP2007}).

\begin{lemma}\label{le:quattro}
Let $s\in[0,2]$. Then
\begin{equation}\label{equi_ma_cla}
\langle\|A\|_{H^1}\rangle^{-m}\|f\|_{H_A^s}\lesssim\|f\|_{H^s}\lesssim \langle\|A\|_{H^1}\rangle^{m}\|f\|_{H_A^s}.
\end{equation}
\end{lemma}

Last, we recall the diamagnetic inequality \cite[Theorem 7.21]{Lieb-Loss-Analysis}, which asserts that for every $A\in L^2_{\mathrm{loc}}(\R^3)$, and $f\in H_A^1(\R^3)$,
\begin{equation}\label{eq:diamag}
|(\nabla|f|)(x)|\leq|(\nabla_Af)(x)|\quad\mbox{for a.e. }x\in\R^3.
\end{equation}
As a particular instance of the diamagnetic inequality, we have the bound
\begin{equation}\label{eq:diamag_senza}
|(\nabla|f|)(x)|\leq|(\nabla f)(x)|\quad\mbox{for a.e. }x\in\R^3,
\end{equation}
for every $f\in H^1(\R^3)$.

\section{A priori estimates}\label{sect:apri}
In this Section we prove suitable a priori estimates for weak solutions to \eqref{eq:MS}, which will play a crucial role in the well-posedness argument. We start with the following result.
\begin{proposition}\label{th:apriori}
Fix $\gamma\in(1,4)$, $s\in[1,2]$, $\sigma>1$, $T>0$, and set $\tilde{\sigma}:=\min\{\sigma,\frac76\}$. Let $(u,A)$ be a weak solution to \eqref{eq:MS}, with $(u,A,\partial_t A)\in L_T^{\infty}M^{s,1}$ and with initial data $(u_0,A_0,A_1)\in M^{s,\sigma}$. Then the following estimates hold true.
\begin{equation}\label{eq:apriori_uno}
\begin{split}
&\|A\|_{L_T^2L^{\infty}}+\|(A,\partial_t A)\|_{L_T^{\infty}\Sigma^{\tilde{\sigma}}}\lesssim_{\langle T\rangle^n}\langle\|(u,A)\|_{L_T^{\infty}(H^1\times H^1)}^m\rangle\langle\|(A_0,A_1)\|_{\Sigma^{\sigma}}^m\rangle,
\end{split}
\end{equation}
\begin{equation}\label{eq:apriori_due}
\begin{split}
&\|u\|_{L^2_TW^{s-1/2,6}}\lesssim_{\langle T\rangle^n}\langle\|(u,A)\|_{L_T^{\infty}(H^1\times H^1)}^m\rangle\langle\|(A_0,A_1)\|_{\Sigma^{\sigma}}^m\rangle\|u\|_{L_T^{\infty}H^s}.
\end{split}
\end{equation}
\end{proposition}


Estimates \eqref{eq:apriori_uno} - \eqref{eq:apriori_due} yield a non trivial gain of integrability for weak solutions to the Maxwell-Schr\"odinger system, due to the smoothing estimates presented in the previous Section. Moreover, \eqref{eq:apriori_uno} implies that the magnetic field enjoys a persistence of regularity property. Finally, we also notice that the right hand side of \eqref{eq:apriori_due} is linear in the higher norm, hence this will enable us to infer global bounds for high Sobolev norms of $u$.

\begin{proof}[Proof of Proposition \ref{th:apriori}]
Let $(q,r)$ be a Klein-Gordon admissible pair, with $q\in(2,4]$. Applying Lemma \ref{le:uno} to the equation $(\square+1)A=\mathbb{P}J+A$, we get
\begin{equation}\label{kl_go}
\|A\|_{L_T^qL^r}\lesssim\|(A_0,A_1)\|_{\Sigma^{2/q}}+\|\mathbb{P}J\|_{L_T^{q'}W^{4/q-1,r'}}+\|A\|_{L_T^1H^{2/q-1}}.
\end{equation}
The first and third term in the r.h.s.~of \eqref{kl_go} are easily controlled. Indeed we have the bound
\begin{equation}\label{trivk}
\|(A_0,A_1)\|_{\Sigma^{2/q}}+\|A\|_{L_T^1H^{2/q-1}}\lesssim \|(A_0,A_1)\|_{\Sigma^{\sigma}}+T\|A\|_{L_T^{\infty}H^1}.
\end{equation}
We focus on the second term. Lemma \ref{le:tre} and the fractional Leibniz rule \eqref{frac_leibniz} yield
\begin{equation}\label{eq:zc}
\begin{split}
\|\mathbb{P}J&\|_{L_T^{q'}W^{4/q-1,r'}}\lesssim_{\langle T\rangle^n} \|\mathbb{P}(\overline{u}\nabla u)+A|u|^2\|_{L_T^{\infty}W^{4/q-1,r'}}\\
&\lesssim_{\langle T\rangle^n} \|u\|_{L_T^{\infty}W^{4/q-1,q}}\|\nabla u\|_{L_T^{\infty}L^2}\\
&\quad+\|A\|_{L_T^{\infty}W^{4/q-1,q}}\|u^2\|_{L_T^{\infty}L^2}+\|u\|_{L_T^{\infty}W^{4/q-1,q}}\|u\|_{L_T^{\infty}L^4}\|A\|_{L_T^{\infty}L^4}\\
&\lesssim_{\langle T\rangle^n} \|u\|_{L_T^{\infty}H^1}^2\langle\|A\|_{L_T^{\infty}H^1}\rangle.
\end{split}
\end{equation}
Combining \eqref{kl_go}, \eqref{trivk} and \eqref{eq:zc}, we deduce that
\begin{equation}\label{eq:cirga}
\|A\|_{L_T^qL_x^r}\lesssim_{\langle T\rangle^n} \langle\|u\|_{L_T^{\infty}H^1}^2\rangle\langle\|A\|_{L_T^{\infty}H^1}\rangle+ \|(A_0,A_1)\|_{\Sigma^1}.
\end{equation}
Interpolating with the trivial bound for $\|A\|_{L_T^{\infty}L^2}$, we conclude that \eqref{eq:cirga} is true for every Klein-Gordon admissible pair $(q,r)$. 

Our next step is to prove that, for any $\delta>0$,
\begin{equation}\label{eq:quasi}
\|u\|_{L_T^2W^{s-1/2-\delta,6}}\lesssim_{\langle T\rangle^n}\langle\|(u,A)\|_{L_T^{\infty}(H^1\times H^1)}^m\rangle\langle\|(A_0,A_1)\|_{\Sigma^{\sigma}}^m\rangle\|u\|_{L_T^{\infty}H^s}.
\end{equation}
Assume for simplicity $\delta\in(0,\frac12)$, and denote by $(q_{\delta},r_{\delta})$ the Klein-Gordon admissible pair such that $r_{\delta}=\frac{3}{2\delta}$. Let us set moreover 
$$p=p(\gamma):=\min\left\{3,\textstyle{\frac{6}{\gamma-1}}\right\}>2,\quad \widetilde{p}=\widetilde{p}(\gamma):=\max\left\{6,\textstyle{\frac{6}{4-\gamma}}\right\}<\infty.$$ 
We consider first the case $\gamma\in(1,3]$. Applying Lemma \ref{le:quattordici} with $\alpha=1/2+\delta$, and using Lemma \ref{esti_conv} and Lemma \ref{esti_pure} for the pair $(p,\widetilde{p})$, we obtain
\begin{align*}
\|u\|_{L_T^2W^{s-1/2-\delta,6}}&\lesssim_{\langle T\rangle^n}\langle\|A\|_{ L_T^{\infty}H^{1}\cap L_T^{q_{\delta}}L^{r_{\delta}}}\rangle^m\|u\|_{L_T^{\infty}H^s}+\|\phi u+|u|^{\gamma-1}u\|_{L_T^2H^{s-1-2\delta}}\\
&\lesssim_{\langle T\rangle^n}\langle\|A\|_{ L_T^{\infty}H^{1}\cap L_T^{q_{\delta}}L^{r_{\delta}}}\rangle^m\|u\|_{L_T^{\infty}H^s}+\|u\|_{L_T^{\infty}H^1}^2\|u\|_{L_T^{\infty}H^s}\\
&\quad+\|u\|_{L_T^{\infty}L^{3(\gamma-1)}}^{\gamma-1}\|u\|_{L_T^{\infty}W^{s-1,6}}\\
&\lesssim_{\langle T\rangle^n}\big(\langle\|A\|_{ L_T^{\infty}H^{1}\cap L_T^{q_{\delta}}L^{r_{\delta}}}\rangle^m+\langle\|u\|_{L_T^{\infty}H^1}\rangle^2\big)\|u\|_{L_T^{\infty}H^s},
\end{align*}
which combined with \eqref{eq:cirga} yield the bound \eqref{eq:quasi}. For the case $\gamma\in(3,4)$ we proceed as follows. Let us set $\delta_1=\frac{\gamma-3}{4}\in(0,\frac14)$, and applying Lemma \ref{le:quattordici}, Lemma \ref{esti_conv}, and Lemma \ref{esti_pure} for the pair $(p,\widetilde{p})$ we get
\begin{equation*}
\begin{split}
\|u&\|_{L_T^2W^{(s-1/2-\delta_1)-,6}}\lesssim_{\langle T\rangle^n}\langle\|A\|_{L_T^{\infty}H^{1}\cap L_T^{q_{\delta_1}}L^{r_{\delta_1}}}\rangle^m\|u\|_{L_T^{\infty}H^s}\\
&+\|\phi u+|u|^{\gamma-1}u\|_{L_T^2H^{(s-1-2\delta_1)-}}\\
&\lesssim_{\langle T\rangle^n}\langle\|A\|_{L_T^{\infty}H^{1}\cap L_T^{q_{\delta_1}}L^{r_{\delta_1}}}\rangle^m\|u\|_{L_T^{\infty}H^s}+\|u\|_{L_T^{\infty}H^1}^2\|u\|_{L_T^{\infty}H^s}\\
&\quad +\|u\|_{L_T^{\infty}L^6}^{\gamma-1}\|u\|_{L_T^{\infty}W^{s-1-2\delta_1},\widetilde{p}}.
\end{split}
\end{equation*}
Combining the bound above, the embedding $H^s\hookrightarrow W^{s-1-2\delta_1,\widetilde{p}}$, and estimate \eqref{eq:cirga} we deduce
\begin{equation}\label{ba_db}
\|u\|_{L_T^2W^{(s-1/2-\delta_1)-,6}}\lesssim_{\langle T\rangle^n}\langle\|(u,A)\|_{L_T^{\infty}(H^1\times H^1)}^m\rangle\langle\|(A_0,A_1)\|_{\Sigma^{\sigma}}^m\rangle\|u\|_{L_T^{\infty}H^s}.
\end{equation}
When $\gamma\in(3,\frac{11}{3}]$, we can apply again Lemma \ref{le:quattordici}, Lemma \ref{esti_conv}, and Lemma \ref{esti_pure} for the pair $(p,\widetilde{p})$, obtaining
\begin{align*}
\|u\|_{L_T^2W^{s-1/2-\delta,6}}&\lesssim_{\langle T\rangle^n}\langle\|A\|_{L_T^{\infty}H^{1}\cap L_T^{q_{\delta}}L^{r_{\delta}}}\rangle^m\|u\|_{L_T^{\infty}H^s}+\|\phi u+|u|^{\gamma-1}u\|_{L_T^2H^{s-1-2\delta}}\\
&\lesssim_{\langle T\rangle^n}\langle\|A\|_{L_T^{\infty}H^{1}\cap L_T^{q_{\delta}}L^{r_{\delta}}}\rangle^m\|u\|_{L_T^{\infty}H^s}+\|u\|_{L_T^{\infty}H^1}^2\|u\|_{L_T^{\infty}H^s}\\
&\quad+\|u\|_{L_T^{\infty}L^{6}}^{\gamma-1}\|u\|_{L_T^{2}W^{s-1,\widetilde{p}}}.
\end{align*}
Combining the bound above with estimates \eqref{eq:cirga} and \eqref{ba_db}, and observing that $W^{s-1/2-\delta_1,6}\hookrightarrow W^{s-1,\widetilde{p}}$, we obtain \eqref{eq:quasi}. In the remaining case $\gamma\in(\frac{11}{3},4)$, we first get a bound for $\|u\|_{L_T^2W^{(s-1/2-\delta_2)-,6}}$, for some $\delta_2\in(0,\delta_1)$, and iterating sufficiently many times such argument we eventually deduce estimate \eqref{eq:quasi}.

Let us fix now a Klein-Gordon admissible pair $(q,r)$ with $r(\tilde{\sigma}-1)>1$, so that $(\tilde{\sigma}-\frac{2}{q})r>3$. Using Sobolev embedding and Lemma \ref{le:uno} we get
\begin{equation}\label{mam_nunv}
\begin{split}
\|A\|_{L_T^2L^{\infty}}+&\|(A,\partial_t A)\|_{L_T^{\infty}\Sigma^{\tilde{\sigma}}}\lesssim \|A\|_{L_{T}^qW^{\tilde{\sigma}-2/q,r}}+\|(A,\partial_t A)\|_{L_T^{\infty}\Sigma^{\tilde{\sigma}}}\\
&\qquad\lesssim \|(A_0,A_1)\|_{\Sigma^{\tilde{\sigma}}}+\|A\|_{L_T^1H^{\tilde{\sigma}-1}}+\|\mathbb{P}J\|_{L_T^{6/5}W^{\tilde{\sigma}-2/3,3/2}}.
\end{split}
\end{equation}
The first two terms in the last expression are easily controlled, indeed we have the bound
\begin{equation}\label{bou_sotri}
\|(A_0,A_1)\|_{\Sigma^{\tilde{\sigma}}}+\|A\|_{L_T^1H^{\tilde{\sigma}-1}}\lesssim  \|(A_0,A_1)\|_{\Sigma^{\sigma}}+T\|A\|_{L_T^{\infty}H^{1}}.
\end{equation}
We focus on the third term, assuming preliminary that $\sigma<\frac76$. Using Lemma \ref{le:tre} and the fractional Leibniz rule \eqref{frac_leibniz}, and observing that $\tilde{\sigma}-2/3<1/2$, we obtain
\begin{equation}\label{dabc}
\begin{split}
\|\mathbb{P}J\|_{L_T^{6/5}W^{\tilde{\sigma}-2/3,3/2}}&\lesssim_{\langle T\rangle^n}\|\mathbb{P}(\overline{u}\nabla u)\|_{L_T^2W^{\tilde{\sigma}-2/3,3/2}} + \|A|u|^2\|_{L_T^{\infty}W^{\tilde{\sigma}-2/3,3/2}}\\
&\lesssim \|u\|_{L_T^2W^{\tilde{\sigma}-2/3,6}}\|\nabla u\|_{L_T^{\infty}L^2}+\|A\|_{L_T^{\infty}W^{1/2,3}}\|u^2\|_{L_T^{\infty}L^3}\\
&\quad+\|A\|_{L_T^{\infty}L^6}\|u\|_{L_T^{\infty}W^{1/2,3}}\|u\|_{L_T^{\infty}L^6}\\
&\lesssim (\|u\|_{L_T^2W^{1/2-\delta,6}}+\|A\|_{L_T^{\infty}H^1})\langle\|u\|_{L_T^{\infty}H^1}^2\rangle,
\end{split}
\end{equation}
for some $\delta>0$ small enough. Combining \eqref{mam_nunv}, \eqref{bou_sotri} and \eqref{dabc} with the bound \eqref{eq:quasi}, we deduce the a priori estimate \eqref{eq:apriori_uno}, in the regime $\sigma<\frac76$.

Next, we apply Lemma \ref{le:quattordici} with $\alpha=1/2$, Lemma \ref{esti_conv}, and Lemma \ref{esti_pure} for the pair $(p,\widetilde{p})$, obtaining
\begin{equation*}
\begin{split}
\|u\|_{L_T^2W^{s-1/2,6}}&\lesssim_{\langle T\rangle^n}\langle\|A\|_{L_T^{\infty}H^{\tilde{\sigma}}\cap L_{T}^2L^{\infty}}^m\rangle\|u\|_{L_T^{\infty}H^s}+\|\phi u+|u|^{\gamma-1}u\|_{L_T^2H^{s-1}}\\
&\lesssim_{\langle T\rangle^n}\langle\|A\|_{L_T^{\infty}H^{\tilde{\sigma}}\cap L_{T}^2L^{\infty}}^m\rangle\|u\|_{L_T^{\infty}H^s}+\|u\|_{L_T^{\infty}H^1}^2\|u\|_{L_T^{\infty}H^s}\\
&\quad\qquad +\|u\|_{L_T^{\infty}L^6}^{\gamma-1}\|u\|_{L_T^2W^{s-1+\eps,\widetilde{p}}},
\end{split}
\end{equation*}
for every $\eps>0$. In particular, let us choose $\eps$ small enough such that $\eps\widetilde{p}<3$, and set $\delta:=3/\widetilde{p}-\eps>0$. Then the bound above, combined with the Sobolev embedding $W^{s-1/2-\delta,6}\hookrightarrow W^{s-1+\eps,\widetilde{p}}$ and the estimates \eqref{eq:quasi} and \eqref{eq:apriori_uno} yields
\begin{align*}
\|u\|_{L_T^2W^{s-1/2,6}}&\lesssim_{\langle T\rangle^n}\big(\langle\|A\|_{L_T^{\infty}H^{\tilde{\sigma}}\cap L_{T}^2L^{\infty}}^m\rangle+\|u\|_{L_T^{\infty}H^1}^2\big)\|u\|_{L_T^{\infty}H^s}\\
&\qquad\quad+\|u\|_{L_T^{\infty}H^1}^{\gamma-1}\|u\|_{L_T^2W^{s-1/2-\delta,6}}\\
&\lesssim_{\langle T\rangle^n}\langle\|(u,A)\|_{L_T^{\infty}(H^1\times H^1)}^m\rangle\langle\|(A_0,A_1)\|_{\Sigma^{\sigma}}^m\rangle\|u\|_{L_T^{\infty}H^s},
\end{align*}
which proves the a priori estimate \eqref{eq:apriori_due}. Moreover, it follows that \eqref{dabc}, whence also $\eqref{eq:apriori_uno}$, is valid for $\sigma\geq\frac76$, which concludes the proof.
\end{proof}

\begin{remark}\label{re:cou_due}
Observe that Lemma \ref{le:tre} has been crucial in order to obtain estimates \eqref{eq:zc} and \eqref{dabc}, as it allows one to avoid the harmful terms which would appear by directly applying a fractional Leibniz rule. This kind of mechanism does not apply for similar models, such as the Maxwell-Pauli system, which involve pure-curl currents modeling quantum spin effects. In fact, local well-posedness at low regularity for the Maxwell-Pauli system appears to be an open problem -- see the discussion in \cite{Kieffer}.
\end{remark}

The a priori estimates encoded in Proposition \ref{th:apriori} turn out to be very useful in the analysis of the Lorentz force associated with a solution $(u,A)$ to the non-linear Maxwell-Schr\"odinger system \eqref{eq:MS}. Let us recall that the Lorentz force is formally defined by $F_L:=\rho E+J\times B$, where $(E,B)$ is the electro-magnetic field, given explicitly by $E=-\partial_t A-\nabla\phi$, $B=\nabla\times A$, by means of the Maxwell equations. 

For sufficiently regular solutions, defined on a time interval $[0,T]$, it is straightforward to deduce that $F_L$ belongs to $L_T^{\infty}L^1$. On the contrary for a generic finite energy  solution it is unknown whether one can give a meaning to the Lorentz force, at least in a distributional sense. As mentioned in the Introduction, this issue already emerges for the classical linear Maxwell-Schr\"odinger system \eqref{clas_MS}. 

As a consequence of estimates \eqref{eq:apriori_uno}-\eqref{eq:apriori_due}, in the next Proposition we show that, as long as $\gamma\in(1,4)$, $u\in H^1$, and the magnetic potential is slightly more regular than being in the energy space, i.e.~$A\in H^{\sigma}$ for some $\sigma>1$, we have that the Lorentz force belongs to a suitable Lebesgue space. 

\begin{proposition}\label{orc}
Fix $\gamma\in(1,4)$, $\sigma\in(1,\frac76)$, and $T>0$. Let $(u,A)$ be a weak solution to the system \eqref{eq:MS}, with $(u,A,\partial_t A)\in L_T^{\infty}M^{1,\sigma}$ and with initial data $(u_0,A_0,A_1)\in M^{1,\sigma}$. Then the Lorentz force $F_L$ belongs to $L_T^2L^1$.
\end{proposition}

\begin{proof}
It is easy to check that $\rho,E\in L_T^{\infty}L^2$, whence $\rho E\in L_T^2L^1$. Moreover, owing to Proposition \ref{th:apriori}, we have $u\in L_T^2W^{1/2,6}$. Since $B\in L_T^{\infty}L^{p}$ for some $p:=p(\sigma)>2$, we deduce that
\begin{equation*}
\begin{split}
\|J\times B\|_{L_T^2L^1}&\lesssim_T\|A\|_{L_T^{\infty}L^6}\|u\|_{L_T^{\infty}L^6}^2\|B\|_{L_T^{\infty}L^2}\\
&\quad+\|u\|_{L_T^{2}L^{2p/(p-2)}}\|\nabla u\|_{L_T^{\infty}L^2}\|B\|_{L_T^{\infty}L^p}\\
&\lesssim\|A\|_{L_T^{\infty}L^6}\|u\|_{L_T^{\infty}L^6}^2\|B\|_{L_T^{\infty}L^2}\\
&\quad+\|u\|_{L_T^{2}W^{1/2,6}}\|\nabla u\|_{L_T^{\infty}L^2}\|B\|_{L_T^{\infty}L^p}\lesssim 1,
\end{split}
\end{equation*}
which concludes the proof.
\end{proof}

\begin{remark}
Given $\gamma\in(1,4)$, and an initial data $(u_0,A_0,A_1)\in M^{1,\sigma}$, with $\sigma\in(1,\frac76)$, the existence of a global weak $M^{1,\sigma}$-solution to \eqref{eq:MS} follows by the existence of a global finite energy weak solution, as proved in \cite{Antonelli-Damico-Marcati}, and the persistence of regularity implied by Proposition \ref{th:apriori}. Moreover, Proposition \ref{orc} guarantees that this solution has a well-defined Lorentz force $F_L\in L^{2}_{\mathrm{loc}}(\R^+;L^1(\R^3))$.
\end{remark}

We conclude this Section with a further a priori estimate for weak $M^{2,\sigma}$-solutions to \eqref{eq:MS}. It will be useful in order to obtain the local well-posedness result for the admissible regime $\sigma\in[\frac43,3)$. Here we do not require linearity in the $H^2$-norm of $u$, in fact we can take an arbitrary $\gamma>1$.

\begin{proposition}\label{th:perdi}
Fix $\gamma>1$, $\sigma\in(1,3)$, and $T>0$. Let $(u,A)$ be a weak solution to \eqref{eq:MS}, with $(u,A,\partial_t A)\in L_T^{\infty}M^{2,1}$ and with initial data $(u_0,A_0,A_1)\in M^{2,\sigma}$. Then $A\in\mathcal{C}([0,T],H^{\sigma})\cap \mathcal{C}^1([0,T],H^{\sigma-1})$, and for every Klein-Gordon admissible pair $(q,r)$ we have the bound
\begin{equation}\label{eq:perdi}
\begin{split}
&\|(A,\partial_tA)\|_{L^q_T(W^{\sigma-2/q,r}\times W^{\sigma-2/q-1,r})}\\
&\qquad\qquad\lesssim_{\langle T\rangle^n}\langle\|A\|_{L_T^{\infty}H^1}^m\rangle\langle\|(A_0,A_1)\|_{\Sigma^{\sigma}}^m\rangle\langle\|u\|_{L_T^{\infty}H^2}^{m}\rangle.
\end{split}
\end{equation}
\end{proposition}

\begin{proof}
Preliminarily we observe that, arguing as in the proof of Theorem \ref{th:apriori}, and allowing all the estimates to depend super-linearly on the $H^2$-norm of $u$, we obtain
\begin{equation}\label{eq:drx}
\|(u,A)\|_{L_T^2(W^{3/2,6}\times L^{\infty})}\lesssim_{\langle T\rangle^n}\langle\|A\|_{L_T^{\infty}(H^1)}^m\rangle\langle\|(A_0,A_1)\|_{\Sigma^{\sigma}}^m\rangle\langle\|u\|_{L_T^{\infty}H^2}^{m}\rangle.
\end{equation}

Next, applying Lemma \ref{le:uno} we get the bound
\begin{equation}\label{nev:end}
\begin{split}
&\|(A,\partial_tA)\|_{L^q_T(W^{\sigma-2/q,r}\times W^{\sigma-2/q-1,r})}\\
&\qquad\qquad\lesssim\|(A_0,A_1)\|_{\Sigma^{\sigma}}+\|A\|_{L_T^1H^{\sigma-1}}+\|\mathbb{P}J\|_{L_T^1H^{\sigma-1}},
\end{split}
\end{equation}
for any Klein-Gordon admissible pair $(q,r)$. 

We assume first $\sigma\in(1,2]$. In this case, the second terms in the r.h.s.~of \eqref{nev:end} is controlled by $\|A\|_{L_T^{\infty}H^1}$. Let us consider the third term. Using Lemma \ref{le:tre} and the fractional Leibniz rule \eqref{frac_leibniz} we get
\begin{equation}\label{te_filo}
\begin{split}
\|\mathbb{P}J\|_{L_T^1H^{\sigma-1}}&\lesssim_{\langle T\rangle^n}\|u\|_{L_T^{\infty}W^{\sigma-1,3}}\|\nabla u\|_{L_T^2L^6}\\
&\qquad\quad+\|A\|_{L_T^{\infty}H^1}\|u\|^2_{L_T^{\infty}L^{\infty}}+\|A\|_{L_T^{\infty}L^3}\|u\|_{L_T^{\infty}W^{1,6}}\|u\|_{L_T^{\infty}L^{\infty}}\\
&\lesssim_{\langle T\rangle^n}\langle\|A\|_{L_T^{\infty}H^1}\rangle\langle\|u\|_{L_T^{\infty}H^2}^2\rangle,
\end{split}
\end{equation}
which combined with \eqref{nev:end} proves the estimate \eqref{eq:perdi}. Moreover, Lemma \ref{le:uno} gives also $A\in\mathcal{C}([0,T],H^{\sigma})\cap \mathcal{C}^1([0,T],H^{\sigma-1})$.

Consider now the case $\sigma\in(2,3)$. Since we have proved already \eqref{eq:perdi} when $\sigma=2$, in particular we get
\begin{equation}\label{mus_stud}
\|A\|_{L_T^{\infty}H^{\sigma-1}}\lesssim\|A\|_{L_T^{\infty}H^{2}}\lesssim_{\langle T\rangle^n}\langle\|A\|_{L_T^{\infty}H^1}^m\rangle\langle\|(A_0,A_1)\|_{\Sigma^{\sigma}}^m\rangle\langle\|u\|_{L_T^{\infty}H^2}^m\rangle.
\end{equation}
Using Lemma \ref{le:tre}, the fractional Leibniz rule \eqref{frac_leibniz}, and estimates \eqref{eq:drx}, \eqref{mus_stud} we deduce
\begin{equation}\label{re_filo}
\begin{split}
\|\mathbb{P}J\|_{L_T^1H^{\sigma-1}}&\lesssim_{\langle T\rangle^n} \|u\|_{L_T^{\infty}W^{\sigma-1,6/(2\sigma-3)}}\|\nabla u\|_{L_T^2L^{3/(3-\sigma)}}\\
&\qquad\quad+\|A\|_{L_T^{\infty}H^{\sigma-1}}\|u\|^2_{L_T^{\infty}L^{\infty}}\\
&\qquad\quad+\|A\|_{L_T^{2}L^{\infty}}\|u\|_{L_T^{\infty}H^{\sigma-1}}\|u\|_{L_T^{\infty}L^{\infty}}\\
&\lesssim_{\langle T\rangle^n} \|u\|_{L_T^{\infty}H^2}\|\nabla u\|_{L_T^2W^{3/2,6}}\\
&\qquad\quad+(\|A\|_{L_T^{\infty}H^{\sigma-1}}+\|A\|_{L_T^{2}L^{\infty}})\|u\|^2_{L_T^{\infty}H^2}\\
&\lesssim_{\langle T\rangle^n}\langle\|A\|_{L_T^{\infty}H^1}^m\rangle\langle\|(A_0,A_1)\|_{\Sigma^{\sigma}}^m\rangle\langle\|u\|_{L_T^{\infty}H^2}^m\rangle.
\end{split}
\end{equation}
Combining estimates \eqref{nev:end}, \eqref{mus_stud} and \eqref{re_filo} we deduce \eqref{eq:perdi}. Again, Lemma \ref{le:uno} gives also $A\in\mathcal{C}([0,T],H^{\sigma})\cap \mathcal{C}^1([0,T],H^{\sigma-1})$, which completes the proof.
\end{proof}

\section{Local well-posedness}
In this Section we prove local well-posedness in $M^{2,\sigma}$, $\sigma\in[\frac43,3)$, for the Cauchy problem associated to the non-linear Maxwell-Schr\"odinger system \eqref{eq:MS}, with $\gamma>1$. The proof is based on a fixed point argument, inspired by \cite{Nakamura-Wada-Local,Nakamura-Wada-MaxwSchr_CMP2007}, where the authors studied the solution theory for the classical Maxwell-Schr\"odinger system \eqref{eq:MSC}, and by the recent paper \cite{Antonelli-Damico-Marcati}, where the authors proved local well-posedness in $M^{2,\frac32}$ for \eqref{eq:MS}, when $\gamma>2$. Here, in addition, we implement Kato's idea \cite{Kato-NLS} (see also \cite[Section 4.8]{cazenave}) to differentiate the Schr\"odinger equation once in time, in order to handle also the case $\gamma\in(1,2]$, and then to recover the $H^2$-regularity from the equation. Moreover, we exploit the a priori estimates of Section 3 in order to cover the whole range $\sigma\in[\frac43,3)$.

We state the main result of this Section.

\begin{theorem}[Local well-posedness]\label{th:LWP}
Fix $\gamma>1$ and $\sigma\in[\frac43,3)$. For any given initial data $(u_0,A_0,A_1)\in M^{2,\sigma}$, there exists a maximal time $T_{max}\in(0,+\infty]$, and a unique (maximal) solution $(u,A)$ to \eqref{eq:MS}, with $(u,A,\partial_t A)\in\mathcal{C}([0,T_{max}),M^{2,\sigma})$. Moreover
\begin{itemize}
\item[(i)] The blow-up alternative holds true, i.e., if $T_{max}<\infty$, then
$$\lim_{t\uparrow T_{max}}\|(u,A,\partial_t A)(t)\|_{M^{2,\sigma}}=+\infty.$$
\item[(ii)] There is continuous dependence on the initial data. Namely, the map $(u_0,A_0,A_1)\mapsto T_{max}$ is lower semicontinuous from $M^{2,\sigma}$ to $\R^+$, and for every $T\in(0,T_{max})$ the flow map $(u_0,A_0,A_1)\mapsto (u,A,\partial_t A)$ is continuous from $M^{2,\sigma}$ to $\mathcal{C}([0,T],M^{2,\sigma})$.
\item[(iii)] The charge $\mathcal{Q}(t):=\|u(t)\|_{L_2}^2$, and the energy $\mathcal{E}(t)$ defined by \eqref{def:energy} are conserved, i.e.~$\mathcal{Q}(t)=\mathcal{Q}(0)$ and $\mathcal{E}(t)=\mathcal{E}(0)$ for every $t\in(0,T_{max})$.
\end{itemize}
\end{theorem}

\begin{remark}
A similar well-posedness result holds true in $M^{s,\sigma}$, for $s\in[11/8,2]$ and $\sigma>1$, and for suitable ranges of $\gamma$ \cite[Chapter 2]{Andreas}. As already mentioned in the Introduction, instead, covering the finite energy case $s=\sigma=1$ is a challenging open question. 
\end{remark}

We start our discussion by proving suitable estimates for the solutions to the linear magnetic Schr\"odinger equation.
\begin{lemma}\label{le:linear_propagator}
Let $T>0$ and $A\in L_T^{\infty}H^1\cap W_T^{1,p}L^3$, for some $p\geq 1$, with $\diver A= 0$. Then, for every $t_0\in[0,T]$, $f\in H^2(\R^3)$, and $F\in W_T^{1,1}L^2$, the Cauchy problem
\begin{equation}\label{eq:magn_nh}
\begin{cases}
i\partial_t u=-\Delta_A u + F\\
u(t_0,\cdot)=f.
\end{cases}
\end{equation}
has a unique solution $u\in\mathcal{C}([0,T];H^2)\cap C^{1}([0,T],L^2)$, which satisfies
\begin{equation}\label{da_int}
\begin{split}
&\|u\|_{W_T^{1,\infty}L^2}\lesssim\Big(\langle\|A_0\|_{H^1}^m\rangle\|u_0\|_{H^2}+\|F(0)\|_{L^2}+\|F\|_{W_T^{1,1}L^2}\\
&\qquad+\|\partial_t A\|_{L_T^pL^3}\langle\|A\|_{H^1}^m\rangle\|F\|_{L_T^{p'}L^2}\Big)e^{\|\partial_t A\|_{L_T^1L^3}\langle\|A\|_{L_T^{\infty}H^1}^m\rangle},
\end{split}
\end{equation}
\begin{equation}\label{da_mil}
\begin{split}
&\|u\|_{L_T^{\infty}H^2}\lesssim\Big(\langle\|A_0\|_{H^1}^m\rangle\|u_0\|_{H^2}+\|F(0)\|_{L^2}+\|F\|_{W_T^{1,1}L^2}\\
&\qquad+\|\partial_t A\|_{L_T^pL^3}\langle\|A\|_{H^1}^m\rangle\|F\|_{L_T^{p'}L^2}\Big)\langle\|A\|_{L_T^{\infty}H^1}^m\rangle e^{\|\partial_t A\|_{L_T^1L^3}\langle\|A\|_{L_T^{\infty}H^1}^m\rangle}.
\end{split}
\end{equation}
\end{lemma}

\begin{proof}
We assume in the proof that $A$ and $F$ are smooth enough, in which case the existence of a solution $u\in\mathcal{C}([0,T];H^2)\cap C^{1}([0,T],L^2)$ is guaranteed by Kato's abstract evolution method \cite{Kato-abstract-1,Kato-abstract-2}. Hence, we only need to prove estimates \eqref{da_int} and \eqref{da_mil}. The general case follows by a standard compactness argument (see the proof of \cite[Lemma 3.1]{Nakamura-Wada-Local} for more details).

Multiplying the equation by $\bar{u}$, integrating by parts and using the self-adjointness of $-\Delta_A$ we deduce the bound $\partial_t\|u\|_{L^2}^2\lesssim\|F\|_{L^2}\|u\|_{L^2}$, which implies 
\begin{equation}\label{bou:una}
\|u\|_{L_T^{\infty}L^2}\lesssim\|u_0\|_{L^2}+\|F\|_{L_T^1L^2}.
\end{equation}
Next, we write the equation for $\partial_t u$, which reads
\begin{equation}\label{eq:dtu}
i\partial_t (\partial_t u)=-\Delta_A(\partial_t u)+2i\partial_t A\cdot\nabla_A u+\partial_t F.
\end{equation}
The energy method applied to \eqref{eq:dtu} yields
\begin{equation}\label{ramin}
\begin{split}
\|\partial_t u\|_{L_T^{\infty}L^2}&\lesssim \|(\partial_t u)(0)\|_{L^2}+\|\partial_tA\cdot\nabla_A u+\partial_t F\|_{L_T^1L^2}.\\
&\lesssim \|\Delta_{A(0)}u(0)\|_{L^2}+\|F(0)\|_{L^2}+\|\partial_tA\cdot\nabla_A u+\partial_t F\|_{L_T^1L^2}.
\end{split}
\end{equation}
Using \eqref{bou:una}, \eqref{ramin} and the equivalence of norms \eqref{equi_ma_cla} we get
\begin{equation}\label{da_ric}
\|u\|_{W_T^{1,\infty}L^2}\lesssim\langle\|A_0\|_{H^1}^m\rangle\|u_0\|_{H^2}+\|F(0)\|_{L^2}+\|F\|_{W_T^{1,1}L^2}+\|\partial_tA\cdot\nabla_Au\|_{L_T^1L^2}.
\end{equation}
Moreover, we have
\begin{equation*}
\begin{split}
&\|\partial_tA\cdot\nabla_Au\|_{L_T^1L^2}\lesssim\int_0^T\|\partial_t A\|_{L^3}\|\nabla_A u\|_{L^6}dt\lesssim\int_0^T\|\partial_t A\|_{L^3}\langle\|A\|_{H^1}^m\rangle\|u\|_{H^2}dt\\
&\lesssim\int_0^T\|\partial_t A\|_{L^3}\langle\|A\|_{H^1}^m\rangle\big(\|u\|_{L^2}+\|\partial_t u\|_{L^2}+\|F\|_{L^2}\big)dt\\
&\lesssim\int_0^T\|\partial_t A\|_{L^3}\langle\|A\|_{H^1}^m\rangle(\|u\|_{L^2}+\|\partial_t u\|_{L^2})dt+\|\partial_t A\|_{L_T^pL^3}\langle\|A\|_{H^1}^m\rangle\|F\|_{L^{p'}L^2}.
\end{split}
\end{equation*}
Combining the inequality above with \eqref{da_ric}, and applying the Gr\"onwall inequality we deduce the bound \eqref{da_int}. Using the equivalence of norms \eqref{equi_ma_cla} we also get
\begin{equation*}
\begin{split}
\|u\|_{L_T^{\infty}H^2}&\lesssim\langle\|A\|_{L_T^{\infty}H^1}^m\rangle\|u\|_{L_T^{\infty}H_A^2}\lesssim\langle\|A\|_{L_T^{\infty}H^1}^m\rangle(\|u\|_{W_T^{1,\infty}L^2}+\|F\|_{L_T^{\infty}L^2})\\
&\lesssim\langle\|A\|_{L_T^{\infty}H^1}^m\rangle(\|u\|_{W_T^{1,\infty}L^2}+\|F(0)\|_{L^2}+\|\partial_t F\|_{L_T^{1}L^2}),
\end{split}
\end{equation*}
which combined with \eqref{da_int} yields \eqref{da_mil}.
\end{proof}

As a consequence of Lemma \ref{le:linear_propagator} we can define, for $A\in L_T^{\infty}H^1\cap W_T^{1,1}L^3$, and for every $t,t_0\in[0,T]$, the linear magnetic propagator $\mathcal{U}_A(t,t_0):H^2\to H^2$, by setting $\mathcal{U}_A(t,t_0)f:=u(t,\cdot)$, where $u$ is the solution to the Cauchy problem \eqref{eq:magn_nh} with $F=0$. The propagator $\mathcal{U}_A$ is a strongly continuous two-parameters $H^2$-semigroup, namely, 
\begin{itemize}
\item $\mathcal{U}_A(t,t)=\mathbb{I}$ for every $t\in[0,T]$.
\item $\mathcal{U}_A(t_1,t_3)=\mathcal{U}_A(t_1,t_2)\mathcal{U}_A(t_2,t_3)$, for every $t_1,t_2,t_3\in[0,T]$.
\item For every $f\in H^2$, the flow map $(t_1,t_2)\mapsto\mathcal{U}_A(t_1,t_2)f$ is continuous from $[0,T]^2$ to $H^2$.
\end{itemize}

Moreover, for every $t,t_0\in[0,T]$, we have $\|\mathcal{U}_A(t,t_0)f\|_{L^2}=\|f\|_{L^2}$, which implies that $\mathcal{U}(t,t_0)$ can be extended to a unitary operator on $L^2$. Interpolating with \eqref{da_int}, we deduce that  for every $s\in[0,2]$, $\mathcal{U}_A$ is a strongly continuous two-parameters $H^s$-semigroup, which satisfies the estimate
\begin{equation}\label{gia_int}
\|\mathcal{U}_A(t,t_0)f\|_{L_T^{\infty}H^s}\lesssim\|f\|_{H^s}\langle\|A\|^{2s}_{L_T^{\infty}H^1}\rangle e^{\frac s2\|\partial_t A\|_{L_T^1L^3}}.
\end{equation}

Last, we observe that $u$ is the solution to the inhomogeneous problem \eqref{eq:magn_nh} if and only if it satisfies the integral formula
\begin{equation}
u(t):=\mathcal{U}_A(t,t_0)f-i\int_{t_0}^t\mathcal{U}_A(t,\tau)F(u)(\tau)d\tau,
\end{equation}
as an identity in $L^2$, for every $t\in[0,T]$.

We are ready to prove the local well-posedness result.

\begin{proof}[Proof of Theorem \ref{th:LWP}]
First, we consider the case $\sigma=\frac43$. We are going to prove the existence of a local solution to \eqref{eq:MS} by means of a fixed point argument.

For $T>0$, $R_1,R_2>1$ to be chosen later, consider the space $Z$ defined by
\begin{equation*}
Z:=\left\{(u,A)\,\left|
\begin{array}{l}
u\in W_T^{1,\infty}L^2\cap L_T^{\infty}H^{2}, u(0)=u_0,\\
A\in L_T^{\infty}H^{4/3}\cap W_T^{1,\infty}H^{1/3}\cap W_T^{1,6}L^3,\,\diver A=0,\\
\|u\|_{W_T^{1,\infty}L^2\cap  L_T^{\infty}H^{2}}\leq R_1,\\
\|A\|_{ L_T^{\infty}H^{4/3}\cap W_T^{1,\infty}H^{1/3}\cap W_T^{1,6}L^3}\leq R_2
\end{array}
\right.
\right\},
\end{equation*}

endowed with the distance
\begin{equation*}
d\big((u^{(1)},A^{(1)}),(u^{(2)},A^{(2)})\big):=\|u^{(1)}-u^{(2)}\|_{L_T^{\infty}L^2}+\|A^{(1)}-A^{(2)}\|_{L_T^{4}L^4}.
\end{equation*}
Observe that, for $R_1\geq 2\|u_0\|_{H^2}$, the space $Z$ is non empty, as it contains the constant map $u_0$. Moreover, it is straightforward to see that $(Z,d)$ is a complete metric space. 

Let us fix the initial data $(u_0,A_0,A_1)\in M^{2,\sigma}$. Set $\mathcal{N}(u):=\phi u+|u|^{\gamma-1}u$, and consider the solution map $\Phi:(u,A)\mapsto(v,B)$, where
\begin{equation}\label{duha:u}
v(t):=\mathcal{U}_A(t,0)u_0-i\int_0^t\mathcal{U}_A(t,\tau)\mathcal{N}(u)(\tau)d\tau,
\end{equation}
\begin{equation}\label{duha:B}
B(t):=\cos(t\mathcal{D})A_0+\frac{\sin(t\mathcal{D})}{\mathcal{D}}A_1+\int_0^t \frac{\sin((t-\tau)\mathcal{D})}{\mathcal{D}}(\mathbb{P}J(u,A)+A)(\tau)d\tau.
\end{equation}
First we show that, for suitable choices of $T>0$, $R_1,R_2>1$, $\Phi$ maps $Z$ into itself. To this aim, let us fix $(u,A)\in Z$. Using Sobolev embedding and the Hardy-Littlewood-Sobolev inequality we get
\begin{equation}\label{nu:well}
\|\mathcal{N}(u)\|_{L_T^{\infty}L^2}\lesssim\|u\|_{L_T^{\infty}H^{2}}^3+\|u\|_{L_T^{\infty}H^{2}}^{\gamma}\lesssim  R_1^{m},
\end{equation}
\begin{equation}\label{der:nu:well}
\|\partial_t\,\mathcal{N}(u)\|_{L_T^{\infty}L^2}\lesssim\big(\|u\|_{L_T^{\infty}H^{2}}^2+\|u\|_{L_T^{\infty}H^{2}}^{\gamma-1}\big)\|\partial_t u\|_{L_T^{\infty}L^2}\lesssim  R_1^{m},
\end{equation}
\begin{equation}\label{non_iniz}
\|\mathcal{N}(u)(0)\|_{L_2}\lesssim \langle\|u_0\|_{H_2}^{m}\rangle.
\end{equation}
Using Lemma \ref{le:linear_propagator}, and estimates \eqref{nu:well}, \eqref{der:nu:well} and \eqref{non_iniz} we deduce that
\begin{equation}\label{fp:dtu}
\|v\|_{W_T^{1,\infty}L^2}\leq C_1\big(\langle\|A_0\|_{H^1}^m\rangle\langle\|u_0\|_{H^2}^m\rangle+TR_1^m+T^{5/6}R_1^mR_2^m\big)e^{T^{5/6}R_2^m},
\end{equation}
\begin{equation}\label{fp:hdue}
\|v\|_{L_T^{\infty}H^2}\leq C_1\big(\langle\|A_0\|_{H^1}^m\rangle\langle\|u_0\|_{H^2}^m\rangle+TR_1^m+T^{5/6}R_1^mR_2^m\big)R_2^me^{T^{5/6}R_2^m}.
\end{equation}
for some positive constant $C_1$.

Next, using Lemma \ref{le:uno} we get
\begin{equation}\label{me_fa}
\begin{split}
\|B&\|_{L_T^{\infty}H^{4/3}\cap W_T^{1,\infty}H^{1/3}\cap W_T^{1,6}L^3}\lesssim\|(A_0,A_1)\|_{\Sigma^{4/3}}+\|\mathbb{P}J(u,A)+A\|_{L_T^{1}H^{1/3}}\\
&\qquad\qquad\qquad\lesssim\|(A_0,A_1)\|_{\Sigma^{4/3}}+T\|\mathbb{P}J(u,A)+A\|_{L_T^{\infty}H^{1/3}}.
\end{split}
\end{equation}
Lemma \ref{le:tre} and the fractional Leibniz rule \eqref{frac_leibniz} yield
\begin{equation}\label{no_pi}
\begin{split}
&\|\mathbb{P}J(u,A)\|_{L_T^{\infty}H^{1/3}}\lesssim\|u\|_{L_T^{\infty}W^{1/3,3}}\|u\|_{L_T^{\infty}W^{1,6}}+\|A\|_{L_T^{\infty}W^{1/3,3}}\|u\|^2_{L_T^{\infty}L^{12}}\\
&\qquad+\|A\|_{L_T^{\infty}L^6}\|u\|_{L_T^{\infty}L^6}\|u\|_{L_T^{\infty}W^{1/3,6}}\lesssim \langle\|A\|_{L_T^{\infty}H^{4/3}}\rangle\|u\|_{L_T^{\infty}H^2}^2,
\end{split}
\end{equation}
which combined with \eqref{me_fa} gives
\begin{equation}\label{fp:B}
\|B\|_{L_T^{\infty}H^{4/3}\cap W_T^{1,\infty}H^{1/3}\cap W_T^{1,6}L^3}\leq C_2\big(\|(A_0,A_1)\|_{\Sigma^{4/3}}+TR_1^2R_2\big),
\end{equation}
for some positive constant $C_2$.

Let us choose $R_1:=4C_1\big\langle\|A_0\|_{H^1}^m\rangle\langle\|u_0\|_{H^2}^m\rangle$ and $R_2=2\langle\|(A_0,A_1)\|_{\Sigma^{4/3}}\rangle$. Using \eqref{fp:dtu}, \eqref{fp:hdue} and \eqref{fp:B}, and choosing $T>0$ sufficiently small (depending on $R_1,R_2$), we deduce that $\Phi$ maps $Z$ into itself. Moreover, arguing as in the proof of \cite[Proposition 3.1]{Antonelli-Damico-Marcati} we can prove that, after choosing $T$ possibly smaller (depending on $R_1,R_2$), $\Phi$ is a contraction on $Z$. Therefore, there exist a unique solution $(u,A)\in Z$ to the system \eqref{eq:MS}. Owing to estimates \eqref{nu:well}-\eqref{der:nu:well}, Lemma \ref{le:linear_propagator} yields $u\in\mathcal{C}([0,T],H^2)\,\cap\,\mathcal{C}^1([0,T],L^2)$. Similarly, estimates \eqref{no_pi} and Lemma \ref{le:uno} give $A\in\mathcal{C}([0,T],H^{\sigma})\cap\mathcal{C}^1([0,T],H^{\sigma-1})$. 

Next, observe that the unconditional uniqueness holds true. Suppose indeed that $\tilde{u}\in L_T^{\infty}H^2\cap W_T^{1,\infty}L^2$, $\tilde{A}\in L_T^{\infty}H^{4/3}\cap W_T^{1,\infty}H^{1/3}$ is a weak solution to \eqref{eq:MS}. By means of Proposition \ref{th:perdi}, we have also that $\tilde{A}\in W_T^{1,6}L^3$. Therefore, after choosing $T$ possibly smaller (depending on $R_1,R_2$), we have $(\tilde{u},\tilde{A})\in Z$, whence $(\tilde{u},\tilde{A})=(u,A)$.

Using uniqueness, we can consider the maximal solution, defined on a (maximal) time interval $[0,T_{max})$. The blow-up alternative easily follows from the fact that a lower bound on the local time of existence depends only on the $M^{2,\frac43}$-norm of the initial data.

In the remaining case $\sigma\in(\frac43,3)$, the existence of a unique maximal $M^{2,\sigma}$-solution to \eqref{eq:MS}, as well as the blow-up alternative, follow by the result for $\sigma=\frac43$ and the persistence of regularity implied by Proposition \ref{th:perdi}. 

Let us prove now that for every $\sigma\in[\frac43,3)$ the charge and energy are conserved. Indeed, taking the imaginary part of the identity
$$(i\partial_t u+\Delta_A u-\mathcal{N}(u),u)=0$$
we get $\partial_t\mathcal{Q}=0$, whence the conservation of charge. Similarly, taking the real part of the identity
$$(i\partial_t u+\Delta_A u-\mathcal{N}(u),\partial_t u)=0,$$
we deduce the conservation of energy.

Finally, for every $\sigma\in[\frac43,3)$, the continuous dependence on the initial data can be proved by adapting the argument in \cite[Proposition 3.2]{Antonelli-Damico-Marcati}. The proof is complete.
\end{proof}

\section{Global well-posedness}
This Section is devoted to the proof of our main result, Theorem \ref{th:main}. We start with the following lemma, which shows the finiteness in time of a suitable norm of the solutions to \eqref{eq:MS}, and will play a key role in the globalization argument.

\begin{lemma}\label{ove_pi}
Let $\gamma\in(1,4)$, $\sigma\in[\frac43,3)$. We fix an initial data $(u_0,A_0,A_1)\in M^{2,\sigma}$, and let $(u,A)$ be the maximal solution to \eqref{eq:MS}, with $(u,A,\partial_t A)\in\mathcal{C}([0,T_{max}),M^{2,\sigma})$. Then, for every $T\in(0,T_{max})$ we have the estimate
\begin{equation}\label{uniform_bound}
\|(u,A)\|_{L_T^2(W^{1/2,6}\times L^{\infty})}+\|(u,A,\partial_t A)\|_{L_T^{\infty}M^{1,7/6}}\lesssim_{\langle T\rangle^n} 1.
\end{equation}
\end{lemma}

\begin{proof}
Using the conservation of energy we obtain
\begin{equation}\label{tou_nier}
\|A\|_{L_T^{\infty}H^1}\lesssim\|\nabla A\|_{L_T^{\infty}L^2}+\int_0^T\|(\partial_tA)(t,\cdot)\|_{L^2}dt\lesssim\langle T\rangle.
\end{equation}
Moreover, the conservation of charge and energy, combined with the equivalence of norms \eqref{equi_ma_cla} and estimate \eqref{tou_nier} yields
\begin{equation}\label{mou_nier}
\|u\|^2_{L_T^{\infty}H^1}\lesssim\langle\|A\|_{L_T^{\infty}H^1}^m\rangle\|u\|^2_{L_T^{\infty}H_A^1}\lesssim_{\langle T\rangle^n} 1.
\end{equation}
Applying the a priori bounds \eqref{eq:apriori_uno} and \eqref{eq:apriori_due}, together with estimates \eqref{tou_nier} and \eqref{mou_nier}, we deduce
\begin{equation*}
\begin{split}
&\|(u,A)\|_{L_T^2(W^{1/2,6}\times L^{\infty})}+\|(A,\partial_t A)\|_{L_T^{\infty}\Sigma^{7/6}}\\
&\qquad\qquad\lesssim_{\langle T\rangle^n}\langle\|(u,A)\|_{L_T^{\infty}(H^1\times H^1)}\rangle^m\langle\|(A_0,A_1)\|_{\Sigma^{\sigma}}\rangle^m\lesssim_{\langle T\rangle^n} 1,
\end{split}
\end{equation*}
which combined with \eqref{mou_nier} yields \eqref{uniform_bound}.
\end{proof}

Next, we introduce a suitable (higher order) modified energy. An analogous functional has been used in \cite{P-T-V}, where the authors study the growth of high Sobolev norms of solutions to the non-linear Schr\"odinger equation on compact manifolds. Similar ideas can be found also in \cite{Antonelli-Marcati-Zheng}, where the authors prove the stability of weak solutions to a one-dimensional quantum hydrodynamic system.

For any given $\gamma>1$, we define the following modified energy:
$$\mathcal{E}_2(t):=\int_{\R^3}|\partial_t u|^2-(\gamma-1)|u|^{\gamma-1}|\nabla|u||^2-\frac{\gamma-1}{\gamma}|u|^{2\gamma}dx.$$

Given any solution $(u,A)$ to the system \eqref{eq:MS}, for $\gamma\in(1,4)$, it turns out that the modified energy $\mathcal{E}_2(t)$ is equivalent to $\|u\|_{H_{A}^2}^2$, up to lower order terms. Indeed, we can prove the following lemma.
\begin{lemma}
Let $\gamma\in(1,4)$, $\sigma\in[\frac43,3)$. We fix an initial data $(u_0,A_0,A_1)\in M^{2,\sigma}$, and let $(u,A)$ be the maximal solution to \eqref{eq:MS}, with $(u,A,\partial_t A)\in\mathcal{C}([0,T_{max}),M^{2,\sigma})$. Then, for every $t\in(0,T_{max})$ we have the estimate
\begin{equation}\label{sue_nn}
\big|\,\mathcal{E}_2(t)-\|u\|^2_{H_{A}^2}\,\big|\lesssim_{\langle t\rangle^n}\langle\|u\|_{H^2}\rangle^{c(\gamma)},
\end{equation}
where $c(\gamma):=\max\{1,\frac{\gamma-1}{2}\}\in[1,2)$.
\end{lemma}

\begin{proof}
Let us set $S(t):=\|(\phi+|u|^{\gamma-1})u\|_{L^2}$, and
$$R(t):=(\gamma-1)\|u^{\gamma-1}|\nabla|u||^2\|_{L^1}+\frac{\gamma-1}{\gamma}\|u^{2\gamma}\|_{L^1},$$
so that $\mathcal{E}_2(t)=\|\partial_tu\|_{L^2}^2-R(t)$. We have
\begin{equation}\label{but_via}
\begin{split}
\big|\,\mathcal{E}_2(t)-\|u\|^2_{H_{A}^2}\big|&\leq\big|\|\partial_t u\|_{L^2}^2-\|\Delta_{A}u\|_{L^2}^2\big|+\|u\|_{L^2}^2+R(t)\\
&\lesssim S(t)\big(\|\Delta_{A}u\|_{L^2}+S(t)\big)+\|u\|_{L^2}^2+R(t).
\end{split}
\end{equation}
Using estimates \eqref{eq:diamag_senza} and \eqref{uniform_bound} we get
\begin{equation}\label{splitR}
\begin{split}
R(t)&\lesssim \|u\|_{L^6}^{\gamma-1}\|\nabla u\|_{L^{12/(7-\gamma)}}^2+\|u\|_{H^1}^{(\gamma+3)\wedge 2\gamma}\|u\|_{H^2}^{(\gamma-3)\vee 0}\\
&\lesssim_{\langle t\rangle^n} \|u\|_{H^2}^{(\gamma-1)/2}+\|u\|_{H^2}^{(\gamma-3)\vee 0}\lesssim \langle\|u\|_{H^2}\rangle^{(\gamma-1)/2}.
\end{split}
\end{equation}
Analogously, we can prove
\begin{equation}\label{cos_vib}
S(t)^2\lesssim \|u\|_{H^1}^3+ \|u\|_{H^1}^{(\gamma+3)\wedge 2\gamma}\|u\|_{H^2}^{(\gamma-3)\vee 0}\lesssim_{\langle t\rangle^n} \langle\|u\|_{H^2}\rangle^{(\gamma-3)\vee 0},
\end{equation}
where we used the Hardy-Littlewood-Sobolev and H\"older inequalities to estimate the non-linear term involving the electric potential $\phi$. Moreover, the equivalence of norms \eqref{equi_ma_cla} and estimate \eqref{uniform_bound} yield
\begin{equation}\label{eq:llp}
\|u\|_{H_A^2}\lesssim\langle\|A\|_{H^1}^m\rangle\|u(t)\|_{H^2}\lesssim_{\langle t\rangle^n}\|u\|_{H^2}.
\end{equation}
Combining \eqref{but_via}-\eqref{eq:llp} we deduce estimate \eqref{sue_nn}.
\end{proof}

Next result shows that, when computing the time derivative of the modified energy, we have a gain in spatial derivatives with respect to the standard energy method. Here we need the assumption $\gamma>2$, which guarantees the existence of the (weak) derivative of $\mathcal{E}_2$.

\begin{lemma}
Let $\gamma\in(2,4)$, $\sigma\in[\frac43,3)$. We fix an initial data $(u_0,A_0,A_1)\in M^{2,\sigma}$, and let $(u,A)$ be the maximal solution to \eqref{eq:MS}, with $(u,A,\partial_t A)\in\mathcal{C}([0,T_{max}),M^{2,\sigma})$. Then, for every $T\in(0,T_{max})$,
\begin{equation}\label{eq:deri_modi}
\begin{split}
&\frac{d}{dt}\mathcal{E}_2=\int_{\R^3}4\RE(\partial_t A\cdot\nabla_A u\cdot\overline{\partial_t u})+(\gamma-1)(\gamma-3)|u|^{\gamma-2}\partial_t|u|\cdot|\nabla|u||^2\\
&\;\;+2(\gamma-1)|u|^{\gamma-2}\partial_t|u|\cdot|\nabla_A u|^2+2(\gamma-1)\phi|u|^{\gamma}\partial_t|u|+2\IM(u\,\partial_t\phi\cdot\overline{\partial_t u})dx,
\end{split}
\end{equation}
as an identity between functions in $W^{1,1}(0,T)$.
\end{lemma}

\begin{proof}
We assume that the solution $(u,A)$ is regular enough so that $\mathcal{E}_2\in\mathcal{C}^1(0,T)$, in which case all the computations below are justified. The general case follows by a standard density argument, owing to the a priori estimates \eqref{eq:apriori_due} and \eqref{eq:perdi}.

We start with the following computation.
\begin{equation*}
\begin{split}
\frac{d}{dt}\|\partial_t u\|_{L^2}^2&=2\RE(\partial_t^2 u,\partial_t u)=2\RE(\partial_t(-\Delta_Au+\phi u+|u|^{\gamma-1}u),i\partial_t u)\\
&=2\RE(-\Delta_A\partial_t u,i\partial_t u)+2\RE([\Delta_A,\partial_t]u,i\partial_t u)\\
&\quad+2\RE(\partial_t (\phi u),i\partial_t u)+2\RE(\partial_t (|u|^{\gamma-1}u),i\partial_t u)\\
&=4\RE(\partial_t A\cdot\nabla_A u,\partial_t u)+2\RE(u\partial_t \phi,i\partial_t u)+2\RE(\partial_t (|u|^{\gamma-1})u,i\partial_t u).
\end{split}
\end{equation*}
Next, we observe that
\begin{equation*}
\begin{split}
&2\RE(\partial_t (|u|^{\gamma-1})u,i\partial_t u)=2\RE(\partial_t (|u|^{\gamma-1})u,-\Delta_Au+\phi u+|u|^{\gamma-1}u)\\
&=2\RE(\partial_t (|u|^{\gamma-1})u,-\Delta_Au)+2(\gamma-1)\int_{\R^3}\phi|u|^{\gamma}\partial_t|u|dx+\frac{\gamma-1}{\gamma}\frac{d}{dt}\int_{\R^3}|u|^{2\gamma}dx.
\end{split}
\end{equation*}
Finally, using the identity $2\RE(\bar{u}\Delta_Au)=\Delta(|u|^2)-2|\nabla_A u|^2$ we get
\begin{equation*}
\begin{split}
2\RE(\partial_t (|u|^{\gamma-1})u,-\Delta_A u)&=-(\partial_t|u|^{\gamma-1},2\RE(\bar{u}\Delta_Au))\\
&=-(\partial_t|u|^{\gamma-1},\Delta(|u|^2))+2(\partial_t|u|^{\gamma-1},|\nabla_Au|^2)\\
&=(\gamma-1)\frac{d}{dt}\int_{\R^3}|u|^{\gamma-1}|\nabla|u||^2dx\\
&\quad+(\gamma-1)(\gamma-3)\int_{\R^3}|u|^{\gamma-2}\partial_t|u||\nabla|u||^2dx\\
&\quad+2(\gamma-1)\int_{\R^3}|u|^{\gamma-2}\partial_t|u||\nabla_Au|^2dx,
\end{split}
\end{equation*}
which concludes the proof.
\end{proof}

Now we are ready to prove our main Theorem.

\begin{proof}[Proof of Theorem \ref{th:main}]
The local well-posedness of the Cauchy problem associated to $\eqref{eq:MS}$ has been proved in Theorem \ref{th:LWP}. We are left to show that for every given initial data $(u_0,A_0,A_1)$, the corresponding solution $(u,A)$ can be extended globally in time (i.e.~$T_{max}=\infty$), and that it satisfies the polynomial bound \eqref{pol_bound} when $\gamma\in(2,3)$, and the exponential bound \eqref{exp_bound} when $\gamma=3$.

Preliminarily, observe that 
\begin{equation}\label{eq:rico}
\begin{split}
\|\partial_t u\|_{L_T^{\infty}L^2}&\lesssim\|\Delta u\|_{L_T^{\infty}L^{2}}+\|A\nabla u\|_{L_T^{\infty}L^{2}}+\|(|A|^2+\phi+|u|^{\gamma-1})u\|_{L_T^{\infty}L^{2}}\\
&\lesssim_{\langle T\rangle^n}\langle\|u\|_{L_T^{\infty}H^2}\rangle+\|A\|_{L_T^{\infty}L^3}\|\nabla u\|_{L_T^{\infty}L^6}+\|A^2\|_{L_T^{\infty}L^2}\|u\|_{L_T^{\infty}L^{\infty}}\\
&\lesssim \langle\|A\|_{L_T^{\infty}H^1}\rangle^2\langle\|u\|_{L_T^{\infty}H^2}\rangle\lesssim_{\langle T\rangle^n}\langle\|u\|_{L_T^{\infty}H^2}\rangle,
\end{split}
\end{equation}
where we used \eqref{cos_vib}, Sobolev embedding, and the bound $\|A\|_{L_T^{\infty}H^1}\lesssim_{\langle T\rangle^n} 1$ provided by \eqref{uniform_bound}. Similarly, using \eqref{cos_vib} and Sobolev embedding we obtain
\begin{equation*}
\begin{split}
\|\partial_t u\|_{L_T^{\infty}H^{-1}}&\lesssim\|\Delta u\|_{L_T^{\infty}H^{-1}}+\|2iA\nabla u+|A|^2u\|_{L_T^{\infty}L^{6/5}}+\|(\phi+|u|^{\gamma-1})u\|_{L_T^{\infty}H^{-1}}\\
&\lesssim_{\langle T\rangle^n} \langle\|u\|_{L_T^{\infty}H^1}\rangle+\|A\|_{L_T^{\infty}L^3}\|\nabla u\|_{L_T^{\infty}L^2}+\|A^2\|_{L_T^{\infty}L^{3/2}}\|u\|_{L_T^{\infty}L^6}\\
&\lesssim\langle\|A\|_{L_T^{\infty}H^1}\rangle^2\langle\|u\|_{L_T^{\infty}H^1}\rangle,
\end{split}
\end{equation*}
which in view of the bound \eqref{uniform_bound} yields
\begin{equation}\label{eq:cori}
\|\partial_t u\|_{L_T^{\infty}H^{-1}}\lesssim_{\langle T\rangle^n} 1.
\end{equation}

Let us also state a couple useful estimates for the electric potential $\phi$. First, applying Sobolev embedding and the Hardy-Littlewood-Sobolev inequality, and using the bound \ref{uniform_bound} we get
\begin{equation}\label{phi_uno}
\|\phi\|_{L_T^{\infty}L^{\infty}}\lesssim \|(-\Delta)^{-1}\rho\|_{L_T^{\infty}W^{3/2,6}}\lesssim \|\rho\|_{L_T^{\infty}L^3}\lesssim \|u\|_{L_T^{\infty}H^1}^2\lesssim_{\langle T\rangle^n} 1.
\end{equation}
In addition, using Sobolev embedding together with estimates \eqref{eq:cori} and \ref{uniform_bound} we obtain
\begin{equation}\label{phi_due}
\begin{split}
\|\partial_t\phi\|_{L_T^{1}L^{3}}&\lesssim_{\langle T\rangle^n} \|(-\Delta)^{-1}\partial_t\rho\|_{L_T^{\infty}\dot{H}^{1/2}}\lesssim \|\partial_t\rho\|_{L_T^{\infty}H^{-3/2}}\\
&\lesssim\|u\|_{L_T^{\infty}H^{3/2}}\|\partial_t u\|_{L_T^{\infty}H^{-1}}\lesssim_{\langle T\rangle^n} \|u\|_{L_T^{\infty}H^2}^{1/2}.
\end{split}
\end{equation}
 
Let us start our analysis by considering the case $\gamma\in(1,2]$. Arguing as in the proof of \eqref{le:linear_propagator} (see estimate \eqref{da_ric}), and using the bound \eqref{uniform_bound} we deduce that for every $T\in(0,T_{max})$
\begin{equation}\label{car_ci}
\|u\|_{L_T^{\infty}H^2}\lesssim_{\langle T\rangle^n}\|\mathcal{N}(u)\|_{W_T^{1,1}L^2}+\langle \|\partial_t A\cdot\nabla_A u\|_{L_T^1L^2}\rangle.
\end{equation}
Using Sobolev embedding, the a priori bound \eqref{eq:apriori_due}, and estimate \eqref{uniform_bound} we get
\begin{equation}\label{fisti_uno}
\begin{split}
\|\partial_t A\cdot\nabla_A &u\|_{L_T^1L^2}\lesssim_{\langle T\rangle^n}\|\partial_t A\|_{L_T^{\infty}L^{9/4}}\|\nabla_A u\|_{L_T^2L^{18}}\\
&\lesssim\|\partial_t A\|_{L_T^{\infty}H^{1/6}}(\|u\|_{L_T^2W^{4/3,6}}+\|A\|_{L_T^2L^{\infty}}\|u\|_{L_T^{\infty}L^{18}})\\
&\lesssim_{\langle T\rangle^n}\|\partial_t A\|_{L_T^{\infty}H^{1/6}}\langle\|A\|_{L_T^2L^{\infty}}\rangle\|u\|_{L_T^{\infty}H^{11/6}}\lesssim_{\langle T\rangle^n}\|u\|_{L_T^{\infty}H^2}^{5/6}.
\end{split}
\end{equation}
Combining \eqref{car_ci}, \eqref{fisti_uno}, and using estimates \eqref{cos_vib}, \eqref{eq:rico}, \eqref{phi_uno}, \eqref{phi_due} and \eqref{uniform_bound} we obtain
\begin{equation}\label{schet}
\begin{split}
\|u\|_{L_T^{\infty}H^2}&\lesssim_{\langle T\rangle^n}\langle \|u\|_{L_T^{\infty}H^2}^{5/6}\rangle+\|\mathcal{N}(u)\|_{L_T^1L^2}+\|\partial_t\mathcal{N}(u)\|_{L_T^1L^2}\\
&\lesssim_{\langle T\rangle^n}\langle \|u\|_{L_T^{\infty}H^2}^{5/6}\rangle +  \|\partial_t\phi\|_{L_T^{1}L^3} \|u\|_{L_T^{\infty}L^6}\\
&\qquad\quad +\|\phi\|_{L_T^{\infty}L^{\infty}} \|\partial_t u\|_{L_T^{1}L^2} + \||u|^{\gamma-1}\partial_t u\|_{L_T^1L^2}\\
&\lesssim_{\langle T\rangle^n}\langle \|u\|_{L_T^{\infty}H^2}^{5/6}\rangle+\|u\|_{L_T^{1}H^2} + \||u|^{\gamma-1}\partial_t u\|_{L_T^1L^2}.
\end{split}
\end{equation}
Using \eqref{schet}, \eqref{eq:rico}, and the Brezis-Gallouet-Wainger inequality \eqref{brez} we get
\begin{equation}\label{gr_bg}
\begin{split}
\|u\|_{L_T^{\infty}H^2}&\lesssim_{\langle T\rangle^n}\langle \|u\|_{L_T^{\infty}H^2}^{5/6}\rangle + \int_0^T \|u\|_{L_t^{\infty}H^2}+\|u(t,\cdot)\|_{L^{\infty}}^{\gamma-1}\|\partial_t u\|_{L_t^{\infty}L^2}dt\\
&\lesssim_{\langle T\rangle^n}\langle \|u\|_{L_T^{\infty}H^2}^{5/6}\rangle + \int_0^T \langle \|u(t,\cdot)\|_{L^{\infty}}^{\gamma-1}\rangle\langle\|u\|_{L_t^{\infty}H^2}\rangle dt\\
&\lesssim_{\langle T\rangle^n}\langle \|u\|_{L_T^{\infty}H^2}^{5/6}\rangle+ \int_0^T\langle\|u(t,\cdot)\|_{W^{1/2,6}}^{\gamma-1}\rangle\times\\
&\quad\qquad\times\big\langle\ln^{5(\gamma-1)/6}(e+\|u\|_{L_t^{\infty}H^2})\big\rangle\langle\|u\|_{L_t^{\infty}H^2}\rangle dt.
\end{split}
\end{equation}
Since $\gamma-1\leq 2$, we can use the bound \eqref{uniform_bound} to obtain
\begin{equation}\label{cat_ar}
\int_0^T\langle\|u(t,\cdot)\|_{W^{1/2,6}}^{\gamma-1}\rangle dt\lesssim_{\langle T\rangle^n}\langle\|u\|_{L_T^2W^{1/2,6}}^{\gamma-1}\rangle\lesssim_{\langle T\rangle^n} 1.
\end{equation}
Combining \eqref{gr_bg} and \eqref{cat_ar}, and observing that $\frac56(\gamma-1)<1$, a Gr\"onwall-type argument (more precisely, the Bihari-LaSalle inequality \cite{Bihari}) yields
$$\|u\|_{L_T^{\infty}H^2}\lesssim\exp\exp\,(T^m).$$ 
Owing to Proposition \ref{th:perdi} and estimate \eqref{uniform_bound}, we also get 
$$\|(A,\partial_t A)\|_{L_T^{\infty}\Sigma^{\sigma}}\lesssim\exp\exp\,(T^m).$$
Therefore, it follows by the blow-up alternative that the solution $(u,A)$ can be extended globally.

Consider now the case $\gamma\in(2,3]$. Using estimate \eqref{sue_nn}, and integrating in time the identity \eqref{eq:deri_modi}, we deduce that for every $T\in(0,T_{max})$
\begin{equation}\label{eq:car_li}
\begin{split}
\|u(T&)\|_{H_{A(T)}^2}^2-\|u(0)\|_{H_{A(0)}^2}^2\lesssim_{\langle T\rangle^n}\mathcal{E}_2(T)-\mathcal{E}_2(0)+\langle\|u\|_{L_T^{\infty}H^2}\rangle\\
&\lesssim_{\langle T\rangle^n}\|\partial_t A\cdot\nabla_A u\,\partial_t u\|_{L_T^1L^1}+\||u|^{\gamma-2}\partial_t|u||\nabla|u||^2\|_{L_T^1L^1}\\
&\quad\qquad+\||u|^{\gamma-2}\partial_t|u||\nabla_Au|^2\|_{L_T^1L^1}\\
&\quad\qquad+\|\phi|u|^{\gamma}\partial_t|u|\|_{L_T^1L^1}+\|u\,\partial_t\phi\,\partial_t u\|_{L_T^1L^1}+\langle\|u\|_{L_T^{\infty}H^2}\rangle.
\end{split}
\end{equation}
Using that $\partial_t|u|\leq|\partial_t u|$ for a.e. $x\in\R^3$, we get from estimates \eqref{eq:car_li} and \eqref{eq:rico}
\begin{equation}\label{eq:car_lib}
\begin{split}
\|u(T&)\|_{H_{A(T)}^2}^2-\|u(0)\|_{H_{A(0)}^2}^2\lesssim_{\langle T\rangle^n}\\
&\Big(\|\partial_t A\cdot\nabla_A u\|_{L_T^1L^2}+\||u|^{\gamma-2}|\nabla|u||^2\|_{L_T^1L^2}+\||u|^{\gamma-2}|\nabla_Au|^2\|_{L_T^1L^2}\\
&\quad+\|\phi|u|^{\gamma}\|_{L_T^1L^2}+\|u\,\partial_t\phi\|_{L_T^1L^2}\Big)\langle\|u\|_{L_T^{\infty}H^2}\rangle+\langle\|u\|_{L_T^{\infty}H^2}\rangle\\
&:=(\textrm{I}+\textrm{II}+\textrm{III}+\textrm{IV}+\textrm{V})\langle\|u\|_{L_T^{\infty}H^2}\rangle+\langle\|u\|_{L_T^{\infty}H^2}\rangle.
\end{split}
\end{equation}
The term (I) is estimated by \eqref{fisti_uno}. Let us estimate the terms II -- V.

\textbf{(II)}~Let us fix $\eps=\eps(\gamma):=\frac{3-\gamma}{4}\in[0,\frac12)$. At spatial level, interpolating $W^{3/2-\eps,6}$ with $H^{2-\eps}$, we deduce
\begin{equation}\label{inter_uno}
\|u\|_{W^{3/2,6/(1+4\eps)}}\lesssim\|u\|_{W^{3/2-\eps,6}}^{1-2\eps}\|u\|_{H^{2-\eps}}^{2\eps}.
\end{equation}
Next, an interpolation between $W^{1/2,6}$ and $W^{3/2,6/(1+4\eps)}$, combined with estimate \eqref{inter_uno} yields
\begin{equation}\label{inter_due}
\begin{split}
\|u\|_{W^{1,6/(1+2\eps)}}&\lesssim\|u\|_{W^{1/2,6}}^{1/2}\|u\|_{W^{3/2,6/(1+4\eps)}}^{1/2}\\
&\lesssim\|u\|_{W^{1/2,6}}^{1/2}\|u\|_{W^{3/2-\eps,6}}^{1/2-\eps}\|u\|_{H^{2-\eps}}^{\eps}.
\end{split}
\end{equation}
Using \eqref{inter_due}, H\"older inequality in time, the a priori estimate \eqref{eq:apriori_due}, and the bound \eqref{uniform_bound} we get
\begin{equation}\label{dav_bas}
\begin{split}
\|\nabla u\|^2_{L_T^2L^{6/(1+2\eps)}}&\lesssim\|u\|_{L_T^2W^{1/2,6}}\|u\|_{L_T^2W^{3/2-\eps,6}}^{1-2\eps}\|u\|_{L_T^{\infty}H^{2-\eps}}^{2\eps}\\
&\lesssim_{\langle T\rangle^n}\|u\|_{L_T^2W^{1/2,6}}\|u\|_{L_T^{\infty}H^{2-\eps}}\\
&\lesssim \|u\|_{L_T^2W^{1/2,6}}\|u\|_{L_T^{\infty}H^1}^{\eps}\|u\|_{L_T^{\infty}H^{2}}^{1-\eps}\\
&\lesssim_{\langle T\rangle^n}\|u\|_{L_T^2W^{1/2,6}}\|u\|_{L_T^{\infty}H^{2}}^{1-\eps}.
\end{split}
\end{equation}
Finally, estimate \eqref{dav_bas} and the diamagnetic bound \eqref{eq:diamag_senza} yield
\begin{equation}\label{fisti_due}
\begin{split}
\||u|^{\gamma-2}|\nabla|u||^2\|_{L_T^1L^2}&\lesssim\|u\|_{L_T^{\infty}L^{6}}^{\gamma-2}\|\nabla u\|^2_{L_T^2L^{6/(1+2\eps)}}\\
&\lesssim_{\langle T\rangle^n}\|u\|_{L_T^2W^{1/2,6}}\|u\|_{L_T^{\infty}H^{2}}^{1-\eps}.
\end{split}
\end{equation}

\textbf{(\textrm{III})}~We have
\begin{equation}\label{zta}
\||u|^{\gamma-2}|\nabla_Au|^2\|_{L_T^1L^2}\lesssim\||u|^{\gamma-2}|\nabla u|^2\|_{L_T^1L^2}+\||u|^{\gamma}|A|^2\|_{L_T^1L^2}.
\end{equation}
The first term in the r.h.s.~can be estimated as in \eqref{fisti_due}:
\begin{equation}\label{qufist}
\||u|^{\gamma-2}|\nabla u|^2\|_{L_T^1L^2}\lesssim_{\langle T\rangle^n}\|u\|_{L_T^2W^{1/2,6}}\|u\|_{L_T^{\infty}H^{2}}^{1-\eps}.
\end{equation}
For the second term, the bound \eqref{uniform_bound} yields
\begin{equation}\label{kar}
\||u|^{\gamma}|A|^2\|_{L_T^1L^2}\lesssim \|u\|_{L_T^{\infty}H^1}^{\gamma}\|A\|_{L_T^2L^{\infty}}^2\lesssim_{\langle T\rangle^n} 1.
\end{equation}
Combining \eqref{zta}, \eqref{qufist} and \eqref{kar} we get
\begin{equation}\label{fisti_tre}
\||u|^{\gamma-2}|\nabla_Au|^2\|_{L_T^1L^2}\lesssim_{\langle T\rangle^n}1+\|u\|_{L_T^2W^{1/2,6}}\|u\|_{L_T^{\infty}H^{2}}^{1-\eps}.
\end{equation}

\textbf{(\textrm{IV})}~Using \eqref{phi_uno} and the bound \eqref{uniform_bound} we get
\begin{equation}\label{fisti_quattro}
\|\phi|u|^{\gamma}\|_{L_T^1L^2}\lesssim_{\langle T\rangle^n}\|\phi\|_{L_T^{\infty}L^{\infty}}\|u\|_{L_{T}^{\infty}L^{2\gamma}}^{\gamma}\lesssim_{\langle T\rangle^n}\|u\|_{L_T^{\infty}H^1}^{\gamma}\lesssim_{\langle T\rangle^n} 1.
\end{equation}

\textbf{(V)}~Using \eqref{phi_due} and the bound \eqref{uniform_bound} we obtain
\begin{equation}\label{fisti_cinque}
\begin{split}
\|u\,\partial_t\phi\|_{L_T^1L^2}&\lesssim\|u\|_{L_T^{\infty}L^6}\|\partial_t\phi\|_{L_T^1L^3}\\
&\lesssim_{\langle T\rangle^n}\|u\|_{L_T^{\infty}H^1}\|u\|_{L_T^{\infty}H^2}^{1/2}\lesssim_{\langle T\rangle^n}\|u\|_{L_T^{\infty}H^2}^{1/2}.
\end{split}
\end{equation}

Combining \eqref{eq:car_lib}, \eqref{fisti_uno}, \eqref{fisti_due}, \eqref{fisti_tre},\eqref{fisti_quattro} and \eqref{fisti_cinque} we obtain
$$
\|u\|_{L_T^{\infty}H_{A}^2}^2-\|u(0)\|_{H_{A(0)}^2}^2\lesssim_{\langle T\rangle^n}\langle\|u\|_{L_T^{\infty}H^2}^{11/6}\rangle+\|u\|_{L_T^2W^{1/2,6}}\langle\|u\|_{L_T^{\infty}H^2}^{2-\eps(\gamma)}\rangle.
$$
Using the equivalence of norms \eqref{equi_ma_cla} and the bound \eqref{uniform_bound} we also get
\begin{equation}\label{eq:funda}
\|u\|_{L_T^{\infty}H_{A}^2}^2-\|u(0)\|_{H_{A(0)}^2}^2\lesssim_{\langle T\rangle^n}\langle\|u\|_{L_T^{\infty}H_A^2}^{11/6}\rangle+\|u\|_{L_T^2W^{1/2,6}}\langle\|u\|_{L_T^{\infty}H_A^2}^{2-\eps(\gamma)}\rangle.
\end{equation}
When $\gamma\in(2,3)$, we can use \eqref{eq:funda} and the bound $\|u\|_{L_T^2W^{1/2,6}}\lesssim _{\langle T\rangle^n} 1$ provided by \eqref{uniform_bound}, obtaining
\begin{equation}\label{eq:funda_S}
\|u\|_{L_T^{\infty}H_{A}^2}^2-\|u(0)\|_{H_{A(0)}^2}^2\lesssim_{\langle T\rangle^n}\langle\|u\|_{L_T^{\infty}H_A^2}\rangle^{\max\{2-\eps(\gamma),11/6\}}.
\end{equation}
Since $\eps(\gamma)=\frac{3-\gamma}{4}>0$, we deduce the polynomial bound
\begin{equation*}
\|u\|_{L_T^{\infty}H_A^2}\lesssim\langle T\rangle^{\frac{m}{\eps(\gamma)}}\approx\langle T\rangle^{\frac{m}{3-\gamma}},
\end{equation*}
which in view of  the equivalence of norms \eqref{equi_ma_cla} and estimate \eqref{uniform_bound} also yields
\begin{equation}\label{eq:pol_u}
\|u\|_{L_T^{\infty}H^2}\lesssim\langle T\rangle^{\frac{m}{3-\gamma}}.
\end{equation}
Moreover, using Proposition \ref{th:perdi} and estimates \eqref{uniform_bound} and \eqref{eq:pol_u} we deduce
\begin{equation}\label{eq:pol_A}
\|(A,\partial_t A)\|_{L_T^{\infty}\Sigma^{\sigma}}\lesssim\langle T\rangle^{\frac{m}{3-\gamma}}.
\end{equation}
Combining estimates \eqref{eq:pol_u} and \eqref{eq:pol_A} we obtain
$$\|(u,A,\partial_t A)\|_{L_T^{\infty}M^{2,\sigma}}\lesssim \langle T\rangle^{\frac{N}{3-\gamma}},$$
for a suitable positive constant $N$ independent of $\gamma$ (note indeed that all the bounds we used to derive \eqref{eq:pol_u} and \eqref{eq:pol_A}, including in particular the a priori estimates \eqref{eq:apriori_uno}-\eqref{eq:apriori_due}, are uniform for $\gamma\in(2,3)$). Using the blow-up alternative, we conclude that the solution $(u,A)$ can be extended globally in time, and that it satisfies the polynomial bound \eqref{pol_bound}. 

When $\gamma=3$ we have $\eps(3)=0$, and estimate \eqref{eq:funda_S} does not give any information. Instead, we rely on the following argument. Let us fix an arbitrary $T\in(0,T_{max})$, and rewrite estimate \eqref{eq:funda} for two generic times $T_0,T_1$, with $0\leq T_0<T_1\leq T$:
\begin{equation}\label{eq:funda_gen}
\begin{split}
\|u\|_{L_{T_1}^{\infty}H_{A}^2}^2&\leq \|u(T_0)\|_{H_{A(T_0)}^2}^2+C_1\langle T\rangle^{m_1}\Big(\langle\|u\|_{L_{T_1}^{\infty}H_A^2}^{11/6}\rangle\\
&+\|u\|_{L^2(T_0,T_1);W^{1/2,6})}\langle\|u\|_{L_{T_1}^{\infty}H_A^2}^2\rangle\Big),
\end{split}
\end{equation}
for some positive constants $C_1, m_1$. Moreover, consider the continuous, increasing function $t\to \|u\|^2_{L_t^2W^{1/2,6}}$, $t\in[0,T]$, which in view of \eqref{uniform_bound} is also bounded by $C_2\langle T\rangle^{m_2}$, for some positive constants $C_2,m_2$. Let us set $\delta:=(2C_1\langle T\rangle^{m_1})^{-2}$, and observe that we can find $k\in\N$, with $k\leq 1+\delta^{-1}C_2\langle T\rangle^{m_2}\lesssim\langle T\rangle^m$, and a sequence of times $t_0=0,t_1,\ldots,t_{k-1}, t_k=T$, such that
$$\|u\|_{L^2((t_j,t_{j+1}),W^{1/2,6})}^2=\|u\|_{L^2_{t_{j+1}}W^{1/2,6}}^2-\|u\|_{L^2_{t_{j}}W^{1/2,6}}^2\leq\delta,\quad j=0,\ldots ,k-1,$$
which in view of the expression for $\delta$ yields
\begin{equation}\label{eq:part_bound}
C_1\langle T\rangle^{m_1}\|u\|_{L^2((t_j,t_{j+1}),W^{1/2,6})}\leq \frac12,\quad j=0,\ldots ,k-1.
\end{equation}
Applying \eqref{eq:funda_gen} with $(T_0,T_1)=(t_j,t_{j+1})$, and owing to \eqref{eq:part_bound}, we obtain
\begin{equation*}
\frac{1}{2}\|u\|_{L^{\infty}_{t_{j+1}}H_A^2}^2\leq 1+\|u(t_j)\|_{H_{A(t_j)}^2}^2+C_1\langle T\rangle^{m_1}\langle\|u\|_{L_{t_{j+1}}^{\infty}H_A^2}^{11/6}\rangle,\quad j=0,\ldots ,k-1,
\end{equation*}
which in particular implies
\begin{equation}\label{iterative_double}
\|u\|_{L_{t_{j+1}}^{\infty}H_A^2}\leq C_3\big(\|u(t_j)\|_{H_{A(t_j)}^2}+\langle T\rangle^{6m_1}\big),\quad j=0,\ldots ,k-1,
\end{equation}
for a suitable constant $C_3>1$. Applying iteratively estimate \eqref{iterative_double}, and using that $k\lesssim \langle T\rangle^{m}$, we eventually obtain an exponential bound
$$\|u\|_{L_T^{\infty}H_A^2}\lesssim \exp(T^m).$$
The bound above, together with the equivalence of norms \eqref{equi_ma_cla} and estimate \eqref{uniform_bound} also yields
\begin{equation}\label{eq:exp_u}
\|u\|_{L_T^{\infty}H^2}\lesssim \exp(T^m).
\end{equation}
Moreover, using Proposition \ref{th:perdi} and estimates \eqref{uniform_bound} and \eqref{eq:exp_u} we deduce
\begin{equation}\label{eq:exp_A}
\|(A,\partial_t A)\|_{L_T^{\infty}\Sigma^{\sigma}}\lesssim \exp(T^m).
\end{equation}
As before, combining estimates \eqref{eq:exp_u} and \eqref{eq:exp_A}, and using the blow-up alternative, we conclude that the solution $(u,A)$ can be extended globally in time, and that it satisfies the exponential bound \eqref{exp_bound}. 

The proof is complete.
\end{proof}

\section{Well-posedness in other gauges}\label{sec:lor}
In this final Section we briefly discuss the possibility of studying the nonlinear Maxwell-Schr\"odinger system in the Lorenz gauge, namely by imposing $\d_t\phi+\diver A=0$. More precisely, we briefly present and discuss the main mathematical differences with our study that uses the Coulomb gauge.
\newline
First of all, in the expansion of the magnetic Laplacian $\Delta_A$ the scalar potential given by $-i\diver A=i\d_t\phi$ does not vanish aymore. This extra term can be handled with the magnetic Koch-Tzvetkov estimates, see Lemma \ref{le:quattordici}, by means of a bootstrap argument, similar to the one used to control the power type nonlinearity in the proof of Proposition \ref{th:apriori}. Furthermore, the electric potential now satisfies a wave equation with $\rho=|u|^2$ as a source term, namely $\square\phi=\rho$. This means that, even if $\phi$ is not given anymore by the Poisson equation, it is sufficiently regular, since the source term does not involve derivatives of the order parameter $u$. Finally, while in the Coulomb gauge we only need to control the solenoidal part of the magnetic potential $\mathbb PA$, here we need to control also $\mathbb QA$. On the other hand, notice that in the Lorenz gauge we have $\mathbb QA=(-\Delta)^{-1}\nabla\d_t\phi$, so that the irrotational part of the magnetic potential can be controlled by exploiting the estimates on the electric potential $\phi$.
\newline
In this sense we see that the nonlinear Maxwell-Schr\"odinger system in the Lorenz gauge can be treated in a similar manner, since the system \eqref{eq:MS} retains the main mathematical difficulties.

\def\cprime{$'$}

\end{document}